\documentclass[12pt, reqno]{amsart} % AMS article style
\usepackage{amssymb,amscd,amsfonts,amsbsy}
\usepackage{latexsym}
\usepackage{exscale}
\usepackage{amsmath,amsthm,amsfonts}
\usepackage{mathrsfs}
\usepackage{xcolor} % Enables the creations of colors.
\usepackage[colorlinks=true,linkcolor=blue,citecolor=red,urlcolor=red, backref=page]{hyperref} % Add coloured links
\usepackage{esint} % For \fint symbol
\usepackage{stmaryrd}
\usepackage{pifont}

\calclayout

\allowdisplaybreaks[3] % Allow equation break between pages

\newtheorem{theorem}{Theorem}[section]
\newtheorem{proposition}[theorem]{Proposition}
\newtheorem{lemma}[theorem]{Lemma}

\theoremstyle{definition}
\newtheorem{definition}[theorem]{Definition}

\theoremstyle{remark}
\newtheorem{remark}[theorem]{Remark}

\numberwithin{equation}{section} %numbers equations within sections.

\def\b{\mathfrak{b}}
\def\C{\mathbb{C}}

\def\R{\mathbb{R}}
\def\Rn{{\mathbb{R}^n}}
\def\M{\mathcal{M}}
\def\N{\mathbb{N}}

\def\S{\mathbb{S}}

\def\BMO{\operatorname{BMO}}
\def\VMO{\operatorname{VMO}}
\def\CMO{\operatorname{CMO}}
\def\MO{\operatorname{MO}}

\def\sgn{\operatorname{sgn}}
\def\sup{\operatornamewithlimits{sup}}
\def\inf{\operatornamewithlimits{inf}}
\def\supp{\operatornamewithlimits{supp}}

\begin{document}

\author{Mingming Cao}
\address{Mingming Cao\\
Instituto de Ciencias Matem\'aticas CSIC-UAM-UC3M-UCM\\
Consejo Superior de Investigaciones Cient{\'\i}ficas\\
C/ Nicol\'as Cabrera, 13-15\\
E-28049 Madrid, Spain} \email{mingming.cao@icmat.es}

\author{K\^{o}z\^{o} Yabuta}
\address{K\^{o}z\^{o} Yabuta\\
Research Center for Mathematics and Data Science\\ 
Kwansei Gakuin University\\ 
Gakuen 2-1, Sanda 669-1337\\ 
Japan.} \email{kyabuta3@kwansei.ac.jp} 

\thanks{The first author acknowledges financial support from the Spanish Ministry of Science and Innovation, through the ``Severo Ochoa Programme for Centres of Excellence in R\&D'' (SEV-2015-0554) and from the Spanish National Research Council, through the ``Ayuda extraordinaria a Centros de Excelencia Severo Ochoa'' (20205CEX001). He is also supported by China Postdoctoral Science Foundation (2018M643280). }

\date{September 30, 2020}


% 42B20 Singular and oscillatory integrals (Calder\'on-Zygmund, etc.)
% 42B25 Maximal functions, Littlewood-Paley theory
% 42B30 H^p-spaces
% 42B35 Function spaces arising in harmonic analysis

\title[$\VMO$ spaces associated with Neumann Laplacian]
{$\VMO$ spaces associated with Neumann Laplacian}

\maketitle
\allowdisplaybreaks

%%%%%%%%%%%%%%%%%%%%% ABSTRACT ABSTRACT  ABSTRACT %%%%%%%%%%%%%%%%%%%%%
\begin{abstract}
In this paper, we establish several different characterizations of the vanishing mean oscillation space associated with Neumann Laplacian $\Delta_N$, written $\VMO_{\Delta_N}(\Rn)$. We first describe it with the classical $\VMO(\Rn)$ and certain $\VMO$ on the half-spaces. Then we demonstrate that $\VMO_{\Delta_N}(\Rn)$ is actually $\BMO_{\Delta_N}(\Rn)$-closure of the space of the smooth functions with compact supports. Beyond that, it can be characterized in terms of compact commutators of Riesz transforms and fractional integral operators associated to the Neumann Laplacian. Additionally, by means of the functional analysis, we obtain the duality between certain $\VMO$ and the corresponding Hardy spaces on the half-spaces. Finally, we present an useful approximation for $\BMO$ functions on the space of homogeneous type, which can be applied to our argument and otherwhere. 
\end{abstract}
%%%%%%%%%%%%%%%%%%%%% ABSTRACT ABSTRACT  ABSTRACT %%%%%%%%%%%%%%%%%%%%%

%%%%%%%%%%%%%%%%%%%%%%% SECTION  SECTION SECTION %%%%%%%%%%%%%%%%%%%%%%%
%%%%%%%%%%%%%%%%%%%%%%% SECTION  SECTION SECTION %%%%%%%%%%%%%%%%%%%%%%%
\section{Introduction} 
In 1970s, Coifman and Weiss \cite{CW} introduced a function space of vanishing mean oscillation, denoted by $\VMO(\Rn)$, which was defined by the closure in the $\BMO$ norm of the space of continuous functions with compact support. They then proved that the Hardy space $H^1(\Rn)$ is the dual of $\VMO(\Rn)$. A deeper study of $\VMO(\Rn)$ space was done by Uchiyama \cite{U}. He proved that the $\VMO(\Rn)$ space can be described by the limits of mean oscillation on cubes. Significantly, it was also given a characterization of $\VMO(\Rn)$ via the copmactness of the commutators of singular integrals. To be more specific, let $1<p<\infty$, $R_j$ be the $j$-th Riesz transform on $\Rn$, and $\mathcal{K}(X,Y)$ be the collection of compact operators from Banach space $X$ to Banach space $Y$. Then there holds that 
\begin{equation}\label{e:bK}
b \in \VMO(\Rn) \ \ \text{ if and only if } \quad [b, R_j] \in \mathcal{K}(L^p(\Rn),L^p(\Rn)). 
\end{equation}
Thus, the commutators behave better than just being bounded, which was obtained by Coifman, Rochberg and Weiss \cite{CRW}. We should remind that the $\VMO(\Rn)$ space throughout this article is different from Sarason's \cite{Sa} although the notation 
is the same. 

The $\VMO$ space and compact commutators have attracted one's attention among researchers in PDEs. The compact commutators were proved by Iwaniec and Sbordone \cite{IS} to be an effective tool in solving elliptic equations with $\VMO$ coefficients. Moreover, the smoothing effect of commutators owns other important applications, for example, the compensated compactness \cite{CLMS} and the integrability theory of Jacobians \cite{I}. Recently, the compactness on Morrey spaces in \cite{PS} were applied to discuss fine Morrey and H\"{o}lder regularity of strong solutions to higher-order elliptic and parabolic equations with $\VMO$ coefficients. Afterwards, the relationship between the $\VMO$ space and compact commutators was extended and improved to the more general cases including Morrey space \cite{CDW}, weighted Lebesgue space \cite{CC, WY}, the operators with non-smooth kernels \cite{CCH, CD},  and the recent bilinear singular integrals \cite{BDMT, BT}. It is worth pointing out that all these results are obtained following Uchiyama's approach, especially the Fr\'{e}chet-Kolmogorov theorem and its variations.   

It is well known that to fully understand the $\VMO$ space, one first should comprehend the larger space $\BMO$. The theory of the classical $\BMO$ was established by John and Nirenberg \cite{JN} and generalized by Duong and Yan \cite{DY-1, DY-2} to the function space $\BMO_L(\Rn)$ associated with an operator $L$.  Soon after,  the authors \cite{DDSTY} introduced and characterized the new function space $\VMO_L$ of vanishing mean oscillation associated with the operator $L$ in the context of the theory of tent spaces. As one has seen, the theory of the classical $\BMO$ and $\VMO$ is closely connected to the Laplacian $\Delta$.  On the other hand, the generalization of the operator $L$ brings the new challenges to study the $\VMO_L$ space. As far as we know, there is almost no literature to explore its other properties except for the duality.  Thus, three basic questions arising from \eqref{e:bK} motivate our work: 
\begin{enumerate}
\item[$\bullet$] Question 1: Does \eqref{e:bK} hold for Riesz transforms 
$\nabla L^{-1/2}$ associated with the operator $L$ other than the Laplacian? 
\item[$\bullet$] Question 2: What type of $\VMO_L$ spaces is suitable to 
\eqref{e:bK} for Riesz transforms $\nabla L^{-\frac12}$? 
\item[$\bullet$] Question 3: Are there other new properties for $\VMO_L$? 
\end{enumerate}

Before addressing these questions, let us get a glimpse of the possibility.  If $L$ is the Dirichlet Laplacian $\Delta_{D_+}$ on $\R^n_+$, then the $\BMO_{\Delta_{D_+}}(\R^n_+)$ space cannot be characterized by the boundedness of $[b, \nabla \Delta_{D_+}^{-1/2}]$ (see \cite[Theorem 1.4]{DHLWY}).  This indicates that the equation \eqref{e:bK} does not hold for $\nabla L^{-\frac12}$  in a very general framework. On the other hand,  \eqref{e:bK} holds for certain special operator, for example the Bessel operator $\Delta_{\lambda}$ in \cite{DLMWY}. Furthermore, as we know, the boundedness is prior condition for the compactness. Taking into consideration some research on the Neumann Laplacian $\Delta_N$ \cite{DDSY}  
and the boundedness of commutators of $\nabla \Delta_N^{-1/2}$ in \cite{LW}, we will pay our attention to the Neumann Laplacian $\Delta_N$.  We postpone all the definitions and notation in Section \ref{s:pre}. 

We begin with giving an answer to Question 1. 

%%%%%%%%%%%%%%%%%%%%%%%% Theorem Theorem Theorem %%%%%%%%%%%%%%%%%%%%%%%%
\begin{theorem}\label{t:c-1}
Let $1<p<\infty$ and $j=1,\ldots,n$. Then $b \in \VMO_{\Delta_N}(\Rn)$ if and only if $[b, R_{N,j}]$ is a compact operator on $L^p(\Rn)$.
\end{theorem}
%%%%%%%%%%%%%%%%%%%%%%%% Theorem Theorem Theorem %%%%%%%%%%%%%%%%%%%%%%%%

Our next main result is to indicate that the equation \eqref{e:bK} also holds for the fractional integrals associated with the Neumann Laplacian $\Delta_N$.  

%%%%%%%%%%%%%%%%%%%%%%%% Theorem Theorem Theorem %%%%%%%%%%%%%%%%%%%%%%%%
\begin{theorem}\label{t:c-2}
Let $0 < \alpha < n$, $ 1 < p < q < \infty $ with $\frac{1}{q} = \frac{1}{p} - \frac{\alpha}{n}$. Then $b \in \VMO_{\Delta_N}(\Rn)$ if and only if $[b, \Delta_{N}^{-\alpha/2}]$ is a compact operator from $L^p(\Rn)$ to $L^q(\Rn)$.
\end{theorem}
%%%%%%%%%%%%%%%%%%%%%%%% Theorem Theorem Theorem %%%%%%%%%%%%%%%%%%%%%%%%

Theorems \ref{t:c-1} and \ref{t:c-2} also provide  positive answers to Question 2. Additionally, $\VMO_L$ space is suitable to \eqref{e:bK} for Riesz transforms $\nabla L^{-\frac12}$ when $L$ is the Neumann Laplacian $\Delta_{N_+}$ ($\Delta_{N_-}$) on the upper (lower) half-space.  Actually,  we have established the desired properties for the corresponding $\VMO$ spaces on the half-space in Section \ref{s:VMO}.  The approach in Section \ref{s:compact} is easily modified to the setting of half-spaces. The details are left to the readers. 

Considering Question 3, we first build a bridge between the $\VMO_{\Delta_N}(\Rn)$ and the classical $\VMO$ space. As we will see, it is quite valuable to further study the $\VMO_{\Delta_N}(\Rn)$ space.  
%%%%%%%%%%%%%%%%%%%%%%%% Theorem Theorem Theorem %%%%%%%%%%%%%%%%%%%%%%%%
\begin{theorem}\label{t:VV}
The $\VMO_{\Delta_N}(\Rn)$ space can be characterized in the following way:
\begin{equation*}
\VMO_{\Delta_N}(\Rn) = \big\{f \in \M(\Rn): f_{+,e} \in \VMO(\Rn) \text{ and } f_{-,e} \in \VMO(\Rn) \big\}.
\end{equation*}
Moreover, we have that
\begin{align*}
\|f\|_{\VMO_{\Delta_N}(\Rn)} \simeq \|f_{+,e}\|_{\VMO(\Rn)} + \|f_{-,e}\|_{\VMO(\Rn)}.
\end{align*}
\end{theorem}
%%%%%%%%%%%%%%%%%%%%%%%% Theorem Theorem Theorem %%%%%%%%%%%%%%%%%%%%%%%%

Beyond that, we can understand the $\VMO_{\Delta_N}(\Rn)$ space in the following way.
%%%%%%%%%%%%%%%%%%%%%%%% Theorem Theorem Theorem %%%%%%%%%%%%%%%%%%%%%%%%
\begin{theorem}\label{t:Cc}
The $\VMO_{\Delta_N}(\Rn)$ space is the $\BMO_{\Delta_N}(\Rn)$-closure of $C_c^{\infty}(\Rn)$.
\end{theorem} 
%%%%%%%%%%%%%%%%%%%%%%%% Theorem Theorem Theorem %%%%%%%%%%%%%%%%%%%%%%%%

We also amalyze the other properties, including characterizations, duality and weak*-convergence, of $\VMO_{\Delta_N}(\Rn)$ and associated spaces in Section \ref{s:VMO} and Section \ref{s:dual}. 

Now let us discuss the strategy of the proof. Generally, the proof of \eqref{e:bK}, as well as other known results about the compactness of commutators, makes use of a characterization of precompactness in Lebesgue spaces, which is the so-called 
Fr\'{e}chet-Kolmogorov theorem. Such theorem has been adapted for various spaces  for examples, \cite{CCH},  \cite{CDW}, \cite{CC} and \cite{DLMWY}.  Even so, it seems to be invalid for the Neumann Laplacian $\Delta_N$. One main reason is that 
the smooth properties on $\Rn$ are not enough although the Riesz transforms $\nabla \Delta_N$ are Calder\'{o}n-Zygmund operators on both $\R^n_+$ and $\R^n_-$. In order to circumvent this obstacle, we reduce our question to that in $L^p_e(\Rn)$, 
which is a closed subspace of $L^p(\Rn)$ and contains all even functions with respect to the last variable.  Theorem \ref{t:VV} is based on the reflection argument on $\Rn$. Thus it allows us to focus on the analysis on half-spaces. The proof of Theorem 
\ref{t:Cc} is constructive but different from Uchiyama's. We mainly apply some $\BMO$ estimates for smooth functions with compact support. In view of Theorem \ref{t:VV}, it needs to connect the functions on the upper and lower spaces by continuity and smoothness. As we mentioned above,  the $\VMO_{\Delta_N}(\Rn)$ space is closely related to those on half-spaces, such as $\VMO_e(\R^n_+)$, $\VMO_r(\R^n_+)$ and $\VMO_z(\R^n_+)$.  Hence, we also investigate their duality to understand $\VMO_{\Delta_N}(\Rn)$ well.  Our method is motivated by \cite{CW} and \cite{C}. Some results from functional analysis is quite effective on our conclusion.  Not only that, we utilize an approximation for $\BMO$ functions by the continuous functions with bounded support. The general case will be presented in Section \ref{s:homo}.

This article is organized as follows. In Section \ref{s:pre}, we recall the definitions of the Neumann Laplacian $\Delta_{N_+}$ and the reflection Neumann Laplacian $\Delta_N$. We also collect some known results related to various types of $\BMO$ spaces. In Section \ref{s:VMO}, we introduce the vanishing mean oscillation space $\VMO_{\Delta_N}(\Rn)$ associated with $\Delta_N$, and provide its characterizations by means of the classical $\VMO(\Rn)$ space, the $\VMO$ on the half-spaces, and smooth functions with compact supports. Section \ref{s:dual} is devoted to the duality between certain $\VMO$ spaces and the corresponding Hardy spaces.  After that, in Section \ref{s:compact}, we establish other characterizations of $\VMO_{\Delta_N}(\Rn)$ using the compact commutators of Riesz transforms and fractional integral operators associated with $\Delta_N$. Finally, in Section \ref{s:homo}, an approximation is presented for $\BMO$ functions on the space of homogeneous type in the sense of Coifman-Weiss.

%%%%%%%%%%%%%%%%%%%%%%% SECTION  SECTION SECTION %%%%%%%%%%%%%%%%%%%%%%%
%%%%%%%%%%%%%%%%%%%%%%% SECTION  SECTION SECTION %%%%%%%%%%%%%%%%%%%%%%%
\section{Preliminaries}\label{s:pre}

%%%%%%%%%%%%%%%%%%%%%%% Subsection Subsection Subsection %%%%%%%%%%%%%%%%%%%%%%
\subsection{The Neumann Laplacian}
The Neumann problem on the half line $(0,\infty)$ is given by the following:
\begin{equation}\label{e:N}
\begin{cases}
u_t - u_{xx} = 0, & \ x,t \in (0,\infty), \\
u(x,0) = \phi(x), & \\
u_{x}(0,t) = 0.
\end{cases}
\end{equation}
Let $\Delta_{1,N_{+}}$ be the Laplacian corresponding to \eqref{e:N}. According to \cite[Section~3.1]{S}, we see that
\begin{align*}
u(x,t) = e^{- t \Delta_{1,N_{+}}} (\phi)(x).
\end{align*}
For $n > 1$, write $\R^n_{+} = \R^{n-1} \times \R_{+}$. And we define the Neumann Laplacian on $\R^{n}_{+}$ by
\begin{align*}
\Delta_{N_+} := \Delta_{n,N_+} = \Delta_{n-1} + \Delta_{1,N_{+}},
\end{align*}
where $\Delta_{n-1}$ is the Laplacian on $\R^{n-1}$. Similarly, we can define Neumann Laplacian $\Delta_{N_-}:=\Delta_{n,N_{-}}$ on $\R^{n}_{-}$.

The Laplacian $\Delta$ and Neumann Laplacian $\Delta_{N_{\pm}}$ are positive definite self-adjoint operators. By the spectral theorem one can define the semigroups generated by these operators $\{e^{-t \Delta}\}_{t \geq 0}$ and $\{e^{-t \Delta_{N_{\pm}}}\}_{t \geq 0}$. Set $p_t(x,y)$ and $p_{t,\Delta_{N_{\pm}}}(x,y)$ to be the heat kernels corresponding to the semigroups generated by $\Delta$ and $\Delta_{N_{\pm}}$, respectively. Then there holds
\[
p_t(x,y) = (4\pi t)^{-\frac{n}{2}} e^{-\frac{|x-y|^2}{4t}}.
\]
It follows from the reflection method \cite[p.~60]{S} that
\begin{align*}
p_{t,\Delta_{N_+}}(x,y) &= (4\pi t)^{-\frac{n}{2}} e^{-\frac{|x'-y'|^2}{4t}}
\Big(e^{-\frac{|x_n-y_n|^2}{4t}} + e^{-\frac{|x_n+y_n|^2}{4t}}\Big), \ \ x,y \in \R^n_+;
\\ %%%%%%%%%%%%%%%
p_{t,\Delta_{N_-}}(x,y) &= (4\pi t)^{-\frac{n}{2}} e^{-\frac{|x'-y'|^2}{4t}}
\Big(e^{-\frac{|x_n-y_n|^2}{4t}} + e^{-\frac{|x_n+y_n|^2}{4t}}\Big), \ \ x,y \in \R^n_-.
\end{align*}

Now let $\Delta_N$ be the uniquely determined unbounded operator acting on $L^2(\Rn)$ such that
\begin{equation}\label{e:NL}
(\Delta_N f)_+ = \Delta_{N_+} f_+ \ \text{ and } \
(\Delta_N f)_- = \Delta_{N_-} f_-
\end{equation}
for all $f:\Rn \to \R$ such that $f_+ \in W^{1,2}(\R^n_+)$ and $f_- \in W^{1,2}(\R^n_-)$. Then $\Delta_N$ is a positive self-adjoint operator and
\begin{align}\label{e:exp}
\big(e^{-t\Delta_N}f \big)_+ = e^{-t \Delta_{N_+}} f_+ \ \ \text{ and } \ \
\big(e^{-t\Delta_N}f \big)_- = e^{-t \Delta_{N_-}} f_-.
\end{align}
The heat kernel of $e^{-t \Delta_N}$, denoted by $p_{t,\Delta_N}(x,y)$, is given by
\begin{align*}
p_{t,\Delta_N}(x,y) = (4\pi t)^{-\frac{n}{2}} e^{-\frac{|x'-y'|^2}{4t}}
\Big(e^{-\frac{|x_n-y_n|^2}{4t}} + e^{-\frac{|x_n+y_n|^2}{4t}}\Big) H(x_n y_n),
\end{align*}
where $H: \R \to \{0,1\}$ is the Heaviside function given by
\[
H(t)=1,\ \text{ if } t \geq 0;\ \  H(t)=0,\ \text{ if } t<0.
\]

Note that
\begin{enumerate}
\item[$\bullet$] The operators $\Delta$, $\Delta_{N_{\pm}}$ and $\Delta_N$ are self-adjoint and they generate bounded analytic positive semigroups acting on all $L^p(\Rn)$ spaces for $1\leq p \leq \infty$; 
\item[$\bullet$] Let $p_{t,L}(x,y)$ be the kernel corresponding to the semigroup generated by one of the operators $L$ listed above. Then $p_{t,L}(x,y)$ satisfies Gaussian bounds:
\[
|p_{t,L}(x,y)| \lesssim t^{-\frac{n}{2}} e^{-\frac{|x-y|^2}{t}},
\]
for all $x,y \in \Omega$, where $\Omega=\Rn$ for $\Delta$ and $\Delta_N$; $\Omega=\R^n_+$ for $\Delta_{N_+}$ and $\Omega=\R^n_-$ for $\Delta_{N_-}$.
\end{enumerate}

The heat kernels for $\Delta_{N_{\pm}}$ and $\Delta_N$ enjoy the smoothness property as follows.
%%%%%%%%%%%%%%%%%%%%%%% Proposition Proposition Proposition %%%%%%%%%%%%%%%%%%%%%
\begin{proposition}[\cite{LW}]
For the operator $L \in \{\Delta_{N_+}, \Delta_{N_-}, \Delta_N\}$, there hold that
\begin{align*}
|p_{t,L}(x,y)-p_{t,L}(x',y)| \lesssim \frac{|x-x'|}{\sqrt{t}+|x-y|}
\frac{\sqrt{t}}{(\sqrt{t}+|x-y|)^{n+1}}
\end{align*}
for any $x,x',y \in \R^n_+$ (or $x,x',y \in \R^n_-$) with $|x-x'|\leq \frac12 |x-y|$;
\begin{align*}
|p_{t,L}(x,y)-p_{t,L}(x,y')| \lesssim \frac{|y-y'|}{\sqrt{t}+|x-y|}
\frac{\sqrt{t}}{(\sqrt{t}+|x-y|)^{n+1}}
\end{align*}
for any $x,y,y' \in \R^n_+$ (or $x,y,y' \in \R^n_-$) with $|y-y'|\leq \frac12 |x-y|$.
\end{proposition}
%%%%%%%%%%%%%%%%%%%%%%% Proposition Proposition Proposition %%%%%%%%%%%%%%%%%%%%%

%%%%%%%%%%%%%%%%%%%%%%%% Subsection Subsection Subsection %%%%%%%%%%%%%%%%%%%%%
\subsection{$\BMO$ spaces}

A locally integrable function $f$ on $\Rn$ is said to be in $\BMO(\Rn)$ if
\[
\|f\|_{\BMO(\Rn)} := \sup_{Q \subseteq \Rn} \frac{1}{|Q|}
\int_{Q} |f(x)-f_Q| dx < \infty,
\]
where $f_Q$ denotes the average value of $f$ on the cube $Q$.

Let $\VMO(\Rn)$ denote the closure of $C_c^{\infty}(\Rn)$ in $\BMO(\Rn)$. Additionally, the space $\VMO(\Rn)$ is endowed with the norm of $\BMO(\Rn)$. Here $C_c^{\infty}(\Rn)$ is the collection of $C^{\infty}(\Rn)$ functions with compact supports.

%%%%%%%%%%%%%%%%%%%%%% Proposition Proposition Proposition %%%%%%%%%%%%%%%%%%%%%%
\begin{proposition}[\cite{U, DDSTY}]\label{p:VMO}
Let $f \in \BMO(\Rn)$. Then $f \in \VMO(\Rn)$ if and only if $f$ satisfies the following three conditions:
\begin{enumerate}
\item[(a)] $\gamma_1(f):=\lim\limits_{r \to 0} \sup\limits_{Q:\ell(Q) \leq r}
\Big(\fint_{Q} |f(x)-f_Q|^2 dx \Big)^{1/2} =0$,
\item[(b)] $\gamma_2(f):=\lim\limits_{r \to \infty} \sup\limits_{Q:\ell(Q) \geq r}
\Big(\fint_{Q} |f(x)-f_Q|^2 dx \Big)^{1/2} =0$,
\item[(c)] $\gamma_3(f):=\lim\limits_{r \to \infty} \sup\limits_{Q \subset Q(0,r)^c}
\Big(\fint_{Q} |f(x)-f_Q|^2 dx \Big)^{1/2} =0$.
\end{enumerate}
\end{proposition}
%%%%%%%%%%%%%%%%%%%%%% Proposition Proposition Proposition %%%%%%%%%%%%%%%%%%%%%%

Let us introduce some notation. For any subset $A \subset \Rn$ and a function $f: \Rn \to \C$, denote by $f|_A$ the restriction of $f$ to $A$. For any function $f$ on $\Rn$, we set
\[
f_+ = f|_{\R^n_+} \ \text{ and } \ f_- = f|_{\R^n_-}.
\]
For any $x=(x',x_n) \in \Rn$ we set $\widetilde{x}=(x',-x_n)$. If $f$ is a function defined on $\R^n_+$, its even extension and 
zero extension defined on $\Rn$ are respectively given by 
\begin{equation*}
f_e(x) :=
\begin{cases}
f(x), & \text{ if } x \in \R^n_+, \\
f(\widetilde{x}), & \text{ if } x \in \R^n_-,
\end{cases}
\ \ \ \ 
f_z(x) :=
\begin{cases}
f(x), & \text{ if } x \in \R^n_+, \\
0, & \text{ if } x \in \R^n_-. 
\end{cases}
\end{equation*}

%%%%%%%%%%%%%%%%%%%%% DEFINITION DEFINITION DEFINITION %%%%%%%%%%%%%%%%%%%%
\begin{definition}
Let $f$ be a function on $\R^n_+$.
\begin{enumerate}
\item[(1)]  $f$ is said to be in $\BMO_r(\R^n_+)$ if there exists
$F\in \BMO(\Rn)$ such that $F|_{\R^n_+} = f$. If $f \in \BMO_r(\R^n_+)$, we set
$\|f\|_{\BMO_r(\R^n_+)} :=\inf \big\{\|F\|_{\BMO(\Rn)}: F|_{\R^n_+} = f\big\}$.
\item[(2)] $f$ is said to be in $\BMO_z(\R^n_+)$ if its zero extension $f_z$ 
belongs to $\BMO(\Rn)$. If $f \in \BMO_z(\R^n_+)$, we set
$\|f\|_{\BMO_z(\R^n_+)} := \|f_z\|_{\BMO(\Rn)}$.
\item[(3)] $f$ is said to be $\BMO_e(\R^n_+)$ if $f_e \in \BMO(\Rn)$.
Moreover, $\BMO_e(\R^n_+)$ is endowed with the norm
$\|f\|_{\BMO_e(\R^n_+)} := \|f_e\|_{\BMO(\Rn)}$.
\end{enumerate}

Similarly one can define the spaces $\BMO_r(\R^n_-)$, $\BMO_z(\R^n_-)$ and $\BMO_e(\R^n_-)$.
\end{definition}
%%%%%%%%%%%%%%%%%%%%% DEFINITION DEFINITION DEFINITION %%%%%%%%%%%%%%%%%%%%

Suppose that $\Omega$ is an open subset of $\Rn$. Define
\begin{align*}
\M(\Omega) :=\bigg\{f \in L^1_{loc}(\Omega): \exists \epsilon>0 \text{ s.t. }
\int_{\Omega} \frac{|f(x)|^2}{1+|x|^{n+\epsilon}}dx < \infty \bigg\}.
\end{align*}

%%%%%%%%%%%%%%%%%%%%% DEFINITION DEFINITION DEFINITION %%%%%%%%%%%%%%%%%%%%
\begin{definition}
We say that $f \in \M(\Omega)$ is of bounded mean oscillation associated with an operator $L$ (abbreviated as $\BMO_L(\Omega)$) if
\begin{align*}
\|f\|_{\BMO_{L}(\Omega)} := \sup_{Q} \fint_{Q} \big| f(x)-e^{-\ell(Q)^2 L} f(x) \big| dx < \infty,
\end{align*}
where the supremum is taken over all cubes $Q$ in $\Omega$.
\end{definition}
%%%%%%%%%%%%%%%%%%%%% DEFINITION DEFINITION DEFINITION %%%%%%%%%%%%%%%%%%%%

The different type $\BMO$ spaces enjoy the following properties.
%%%%%%%%%%%%%%%%%%%% PROPOSITION PROPOSITION PROPOSITION %%%%%%%%%%%%%%%%%%
\begin{proposition}[\cite{DDSY}]\label{p:BMO}
There hold that
\begin{align*}
\|f\|_{\BMO_{\Delta_{N_+}}(\R^n_+)} &\simeq \|f\|_{\BMO_e(\R^n_+)} \simeq \|f\|_{\BMO_r(\R^n_+)}, \\
\|f\|_{\BMO_{\Delta_{N_-}}(\R^n_-)} &\simeq \|f\|_{\BMO_e(\R^n_-)} \simeq \|f\|_{\BMO_r(\R^n_-)}, \\
\|f\|_{\BMO_{\Delta_N}(\Rn)} &\simeq \|f_{+,e}\|_{\BMO(\Rn)} + \|f_{-,e}\|_{\BMO(\Rn)}.
\end{align*}
\end{proposition}
%%%%%%%%%%%%%%%%%%%% PROPOSITION PROPOSITION PROPOSITION %%%%%%%%%%%%%%%%%%

Additionally, the authors in \cite{AR} and \cite{ART} further investigated the $\BMO_r(\Omega)$, $\BMO_z(\Omega)$ and corresponding Hardy spaces if $\Omega$ is a Lipschitz domain. The local case can be found in \cite{C}. 

%%%%%%%%%%%%%%%%%%%%%% SECTION  SECTION SECTION %%%%%%%%%%%%%%%%%%%%%%%%
%%%%%%%%%%%%%%%%%%%%%% SECTION  SECTION SECTION %%%%%%%%%%%%%%%%%%%%%%%%
\section{$\VMO_{\Delta_N}(\Rn)$ space}\label{s:VMO}

%%%%%%%%%%%%%%%%%%%%% DEFINITION DEFINITION DEFINITION %%%%%%%%%%%%%%%%%%%%
\begin{definition}
We say that a function $f \in \BMO_L(\Omega)$ belongs to $\VMO_L(\Omega)$, the space of functions of vanishing mean oscillation associated with the semigroup $\{e^{-tL}\}_{t>0}$, if it satisfies the limiting conditions
\begin{align*}
\gamma_1(f;L) &:= \lim_{r \to 0} \sup_{Q \subseteq \Omega:\ell(Q) \leq r}
\bigg(\fint_{Q} |f(x)-e^{-\ell(Q)^2 L} f(x) |^2 dx \bigg)^{1/2}=0,
\\ %%%%%%%%%%%%%%%
\gamma_2(f;L) &:= \lim_{r \to \infty} \sup_{Q\subseteq \Omega:\ell(Q) \geq r}
\bigg(\fint_{Q} |f(x)-e^{-\ell(Q)^2 L} f(x)|^2 dx \bigg)^{1/2}=0,
\\ %%%%%%%%%%%%%%%
\gamma_3(f;L) &:= \lim_{r \to \infty} \sup_{Q \subseteq \Omega \cap Q(0,r)^c}
\bigg(\int_{Q} |f(x)-e^{-\ell(Q)^2 L} f(x)|^2 dx \bigg)^{1/2}=0.
\end{align*}
We endow $\VMO_L(\Omega)$ with the norm of $\BMO_L(\Omega)$.
\end{definition}
%%%%%%%%%%%%%%%%%%%%% DEFINITION DEFINITION DEFINITION %%%%%%%%%%%%%%%%%%%%

One of the main theorems in this section is to establish the relationship between the $\VMO_{\Delta_N}(\Rn)$ space and the classical $\VMO(\Rn)$ space. 
%%%%%%%%%%%%%%%%%%%%%%%%% Theorem Theorem Theorem %%%%%%%%%%%%%%%%%%%%%%%
\begin{theorem}\label{t:VMO}
The $\VMO_{\Delta_N}(\Rn)$ space can be characterized in the following way:
\[
\VMO_{\Delta_N}(\Rn) = \big\{f \in \M(\Rn): f_{+,e} \in \VMO(\Rn)
\text{ and } f_{-,e} \in \VMO(\Rn) \big\}.
\]
Moreover, we have that
\begin{align*}
\|f\|_{\VMO_{\Delta_N}(\Rn)}
\simeq \|f_{+,e}\|_{\VMO(\Rn)} + \|f_{-,e}\|_{\VMO(\Rn)}.
\end{align*}
\end{theorem}
%%%%%%%%%%%%%%%%%%%%%%%%% Theorem Theorem Theorem %%%%%%%%%%%%%%%%%%%%%%%

To understand the $\VMO_{\Delta_N}(\Rn)$ space well, let us describe it in terms of $\VMO$ spaces on the upper/lower half-spaces. 
%%%%%%%%%%%%%%%%%%%%%%%%% Theorem Theorem Theorem %%%%%%%%%%%%%%%%%%%%%%%
\begin{theorem}\label{t:VMO-1}
The $\VMO_{\Delta_N}(\Rn)$ space can be described as
\[
\VMO_{\Delta_N}(\Rn) = \big\{f \in \M(\Rn): f_+ \in \VMO_{\Delta_{N_+}}(\R^n_+)
\text{ and } f_- \in \VMO_{\Delta_{N_-}}(\R^n_-) \big\}.
\]
Moreover, we have that
\[
\|f\|_{\VMO_{\Delta_N}(\Rn)} \simeq
\|f_+\|_{\VMO_{\Delta_{N_+}}(\R^n_+)} + \|f_-\|_{\VMO_{\Delta_{N_-}}(\R^n_-)}.
\]
\end{theorem}
%%%%%%%%%%%%%%%%%%%%%%%%% Theorem Theorem Theorem %%%%%%%%%%%%%%%%%%%%%%%

%%%%%%%%%%%%%%%%%%%%%%%%%%% Proof Proof Proof %%%%%%%%%%%%%%%%%%%%%%%%%%
\begin{proof}
Recall that
\begin{align*}
\big(e^{-t\Delta_N}f \big)_+ = e^{-t \Delta_{N_+}} f_+ \ \ \text{ and } \ \
\big(e^{-t\Delta_N}f \big)_- = e^{-t \Delta_{N_-}} f_-.
\end{align*}
Then for any cube $Q \subset \Rn$ we have
\begin{align}\label{e:Q}
&\fint_Q  |f(x)-e^{-\ell(Q)^2 \Delta_N}f(x) \big| dx
\nonumber \\ %%%%%%%%%%%%%%%
&=\frac{1}{|Q|} \int_{Q \cap \R^n_+} |f_+(x)-e^{-\ell(Q)^2 \Delta_{N_+}}f_+(x)| dx
\nonumber \\ %%%%%%%%%%%%%%%
&\quad+ \frac{1}{|Q|}\int_{Q \cap \R^n_-} |f_-(x)-e^{-\ell(Q)^2 \Delta_{N_-}}f_-(x)| dx,
\end{align}
and
\begin{align}\label{e:QQ}
&\bigg(\fint_Q  |f(x)-e^{-\ell(Q)^2 \Delta_N}f(x)|^2 dx\bigg)^{\frac12}
\nonumber\\ %%%%%%%%%%%%%%%
&=\bigg(\frac{1}{|Q|} \int_{Q \cap \R^n_+} |f_+(x)-e^{-\ell(Q)^2 \Delta_{N_+}}f_+(x)|^2 dx
\nonumber\\ %%%%%%%%%%%%%%%
&\quad+ \frac{1}{|Q|}\int_{Q \cap \R^n_-} |f_-(x)-e^{-\ell(Q)^2 \Delta_{N_-}}f_-(x)|^2 dx\bigg)^{\frac12}
\nonumber\\ %%%%%%%%%%%%%%%
&\simeq \bigg(\frac{1}{|Q|} \int_{Q \cap \R^n_+} |f_+(x)-e^{-\ell(Q)^2 \Delta_{N_+}}f_+(x)|^2 dx\bigg)^{\frac12}
\nonumber\\ %%%%%%%%%%%%%%%
&\quad+ \bigg(\frac{1}{|Q|} \int_{Q \cap \R^n_-} |f_-(x)-e^{-\ell(Q)^2 \Delta_{N_-}}f_-(x)|^2 dx\bigg)^{\frac12}.
\end{align}
If $f \in \VMO_{\Delta_N}(\Rn)$, then \eqref{e:Q} and \eqref{e:QQ} respectively indicate that
\[
\|f_+\|_{\BMO_{\Delta_{N_+}}(\R^n_+)} + \|f_-\|_{\BMO_{\Delta_{N_-}}(\R^n_-)}
\lesssim \|f\|_{\BMO_{\Delta_N}(\Rn)},
\]
and
\[
\gamma_i(f_+;\Delta_{N_+}) + \gamma_i(f_-;\Delta_{N_-})
\lesssim \gamma_i(f;\Delta_N)=0, \ \ i=1,2,3.
\]
This shows $f_+ \in \VMO_{\Delta_{N_+}}(\R^n_+)$ and
$f_- \in \VMO_{\Delta_{N_-}}(\R^n_-)$ with
\begin{align}\label{e:V-1}
\|f_+\|_{\VMO_{\Delta_{N_+}}(\R^n_+)} + \|f_-\|_{\VMO_{\Delta_{N_-}}(\R^n_-)}
\lesssim \|f\|_{\VMO_{\Delta_N}(\Rn)}.
\end{align}

Let us demonstrate the another direction. Assume now that $f \in \M(\Rn)$ such that $f_+ \in \VMO_{\Delta_{N_+}}(\R^n_+)$ and $f_- \in \VMO_{\Delta_{N_-}}(\R^n_-)$. Hence, there hold that $f_+ \in \BMO_{\Delta_{N_+}}(\R^n_+)$, $f_- \in \BMO_{\Delta_{N_-}}(\R^n_-)$, and
\[
\gamma_i(f_+;\Delta_{N_+}) = \gamma_i(f_-;\Delta_{N_-})=0, \ \ i=1,2,3.
\]

Let $Q \subset \Rn$ be an arbitrary cube. Let us analyze the position of $Q$. If $Q \subset \R^n_+$, then it yields that
\begin{align*}
\fint_{Q} |f-e^{-\ell(Q)^2 \Delta_N}f| dx
=\fint_{Q} |f_+-e^{-\ell(Q)^2 \Delta_{N_+}}f_+| dx
\leq \|f_+\|_{\BMO_{\Delta_{N_+}}(\R^n_+)},
\end{align*}
and
\begin{align*}
\bigg(\fint_{Q} |f-e^{-\ell(Q)^2 \Delta_N}f|^2 dx\bigg)^{\frac12}
=\bigg(\fint_{Q} |f_+ -e^{-\ell(Q)^2 \Delta_{N_+}}f_+|^2 dx\bigg)^{\frac12}.
\end{align*}
Similarly, if $Q \subset \R^n_-$, we have
\begin{align*}
\fint_{Q} |f-e^{-\ell(Q)^2 \Delta_N}f| dx
\leq \|f_-\|_{\BMO_{\Delta_{N_-}}(\R^n_-)},
\end{align*}
and
\begin{align*}
\bigg(\fint_{Q} |f-e^{-\ell(Q)^2 \Delta_N}f|^2 dx\bigg)^{\frac12}
=\bigg(\fint_{Q} |f_- -e^{-\ell(Q)^2 \Delta_{N_-}}f_-|^2 dx\bigg)^{\frac12}.
\end{align*}

If $Q \cap \R^n_+ \neq \emptyset$ and $Q \cap \R^n_- \neq \emptyset$, then we define
\begin{equation}\label{e:Q+}
\begin{split}
\widehat{Q}_+ &:=\{x=(x',x_n): x' \in Q \cap \R^{n-1},\ 0 < x_n \leq \ell(Q)\},
\\ %%%%%%%%%%%%%%%
\widehat{Q}_- &:=\{x=(x',x_n): x' \in Q \cap \R^{n-1},\ -\ell(Q) < x_n \leq 0\}.
\end{split}
\end{equation}
Observe that $Q \cap \R^n_+ \subset \widehat{Q}_+ \subset \R^n_+$, $Q \cap \R^n_- \subset \widehat{Q}_- \subset \R^n_-$ and
$|\widehat{Q}_+|=|\widehat{Q}_-|=|Q|$. Invoking \eqref{e:Q} and \eqref{e:QQ}, we obtain that
\begin{align*}
\fint_{Q} |f-e^{-\ell(Q)^2 \Delta_N}f| dx
&\leq \frac{1}{|\widehat{Q}_+|} \int_{\widehat{Q}_+}
|f_+ -e^{-\ell(\widehat{Q}_+)^2 \Delta_{N_+}}f_+| dx
\\ %%%%%%%%%%%%%%%
&\quad+ \frac{1}{|\widehat{Q}_-|} \int_{\widehat{Q}_-}
|f_- -e^{-\ell(\widehat{Q}_-)^2 \Delta_{N_-}}f_-| dx
\\ %%%%%%%%%%%%%%%
&\leq \|f_+\|_{\BMO_{\Delta_{N_+}}(\R^n_+)} + \|f_-\|_{\BMO_{\Delta_{N_-}}(\R^n_-)},
\end{align*}
and
\begin{align*}
\bigg(\fint_{Q} |f -e^{-\ell(Q)^2 \Delta_N}f|^2 dx\bigg)^{\frac12}
&\lesssim \bigg(\fint_{\widehat{Q}_+} |f_+ -e^{-\ell(\widehat{Q}_+)^2 \Delta_{N_+}}f_+|^2 dx\bigg)^{\frac12}
\\ %%%%%%%%%%%%%%%
&\quad+ \bigg(\fint_{\widehat{Q}_-} |f_- -e^{-\ell(\widehat{Q}_-)^2 \Delta_{N_-}}f_-|^2 dx\bigg)^{\frac12}.
\end{align*}
Collecting the above estimates, we deduce that
\[
\|f\|_{\BMO_{\Delta_N}(\Rn)} \lesssim
\|f_+\|_{\BMO_{\Delta_{N_+}}(\R^n_+)} + \|f_-\|_{\BMO_{\Delta_{N_-}}(\R^n_-)},
\]
and
\[
\gamma_i(f;\Delta_N) \lesssim
\gamma_i(f_+;\Delta_{N_+}) + \gamma_i(f_-;\Delta_{N_-})=0, \ \ i=1,2,3.
\]
This proves that $f \in \VMO_{\Delta_N}(\Rn)$ with
\begin{align}\label{e:V-2}
\|f\|_{\VMO_{\Delta_N}(\Rn)}
\lesssim \|f_+\|_{\VMO_{\Delta_{N_+}}(\R^n_+)} + \|f_-\|_{\VMO_{\Delta_{N_-}}(\R^n_-)}.
\end{align}

Consequently, the desired result follows from the inequalities \eqref{e:V-1} and \eqref{e:V-2}.
\end{proof}
%%%%%%%%%%%%%%%%%%%%%%%%%% End End End Proof %%%%%%%%%%%%%%%%%%%%%%%%%%

Let us introduce several types of $\VMO$ spaces on the half-spaces. 
%%%%%%%%%%%%%%%%%%%%%% DEFINITION DEFINITION DEFINITION %%%%%%%%%%%%%%%%%%%
\begin{definition}\label{d:VMO}
Let $f$ be a function on $\R^n_+$.
\begin{enumerate}
\item[(1)]  $f$ is said to be in $\VMO_r(\R^n_+)$ if there exists
$F\in \VMO(\Rn)$ such that $F|_{\R^n_+} = f$. If $f \in \VMO_r(\R^n_+)$, we set
$\|f\|_{\VMO_r(\R^n_+)} :=\inf \big\{\|F\|_{\VMO(\Rn)}: F|_{\R^n_+} = f\big\}$.
\item[(2)] $f$ is said to be in $\VMO_z(\R^n_+)$ if the function $f_z$ belongs to
$\VMO(\Rn)$. If $f \in \VMO_z(\R^n_+)$, we set
$\|f\|_{\VMO_z(\R^n_+)} := \|f_z\|_{\VMO(\Rn)}$.
\item[(3)] $f$ is said to be $\VMO_e(\R^n_+)$ if $f_e \in \VMO(\Rn)$.
Moreover, $\VMO_e(\R^n_+)$ is endowed with the norm
$\|f\|_{\VMO_e(\R^n_+)} := \|f_e\|_{\VMO(\Rn)}$.
\end{enumerate}

Similarly one can define the spaces $\VMO_r(\R^n_-)$, $\VMO_z(\R^n_-)$ and $\VMO_e(\R^n_-)$.
\end{definition}
%%%%%%%%%%%%%%%%%%%%%% DEFINITION DEFINITION DEFINITION %%%%%%%%%%%%%%%%%%%

%%%%%%%%%%%%%%%%%%%%%%%% Theorem Theorem Theorem %%%%%%%%%%%%%%%%%%%%%%%%
\begin{theorem}\label{t:VMO-2}
The spaces $\VMO_{\Delta_{N_+}}(\R^n_+)$, $\VMO_e(\R^n_+)$ and $\VMO_r(\R^n_+)$ coincide, with equivalent norms
\[
\|f\|_{\VMO_{\Delta_{N_+}}(\R^n_+)}
\simeq \|f\|_{\VMO_e(\R^n_+)}
\simeq \|f\|_{\VMO_r(\R^n_+)}.
\]
Similar results hold for $\VMO_{\Delta_{N_-}}(\R^n_-)$, $\VMO_e(\R^n_-)$ and $\VMO_r(\R^n_-)$.
\end{theorem}
%%%%%%%%%%%%%%%%%%%%%%%% Theorem Theorem Theorem %%%%%%%%%%%%%%%%%%%%%%%%

%%%%%%%%%%%%%%%%%%%%%%%%%% Proof Proof Proof %%%%%%%%%%%%%%%%%%%%%%%%%%%
\begin{proof}
We will present the proof in two steps.

\smallskip\noindent{\bf Step 1: $\VMO_{\Delta_{N_+}}(\R^n_+)=\VMO_e(\R^n_+)$}. We first present two fundamental identities:
\begin{align}
\label{e:t-1} e^{-t \Delta} f_e(x) &= e^{-t \Delta_{N_+}} f(x), \ x \in \R^n_+, t>0; \\
\label{e:t-2} e^{-t \Delta} f_e(x) &= e^{-t \Delta_{N_+}} f(\widetilde{x}), \ x \in \R^n_-, t>0.
\end{align}
We only prove \eqref{e:t-2}. Note that $p_{t,\Delta_{N_+}}(x,y)=p_t(x,y)+p_t(x,\widetilde{y})$ and
\[
p_t(x,\widetilde{y}) = p_t(\widetilde{x},y), \ \
p_t(\widetilde{x},\widetilde{y}) = p_t(x,y).
\]
Hence, we have for any $x \in \R^n_-$ and $t>0$
\begin{align*}
e^{-t \Delta_{N_+}} f(\widetilde{x})
&=\int_{\R^n_+} p_{t,\Delta_{N_+}}(\widetilde{x},y) f(y) dy
\\ %%%%%%%%%%%%%%%
&=\int_{\R^n_+} p_{t}(\widetilde{x},y) f(y) dy
+ \int_{\R^n_+} p_{t}(\widetilde{x},\widetilde{y}) f(y) dy
\\ %%%%%%%%%%%%%%%
&=\int_{\R^n_+} p_{t}(x,\widetilde{y}) f(y) dy
+ \int_{\R^n_+} p_{t}(x,y) f(y) dy
\\ %%%%%%%%%%%%%%%
&=\int_{\R^n_-} p_{t}(x,y) f(\widetilde{y}) dy
+ \int_{\R^n_+} p_{t}(x,y) f(y) dy
\\ %%%%%%%%%%%%%%%
&= \int_{\Rn} p_{t}(x,y) f_e(y) dy
=e^{-t \Delta} f_e(x).
\end{align*}

Now we show $\VMO_e(\R^n_+) \subseteq \VMO_{\Delta_{N_+}}(\R^n_+)$. Let $f \in \VMO_e(\R^n_+)$. Then it yields that $f_e \in \VMO(\Rn)=\VMO_{\Delta}(\Rn)$, where the equivalence will be proved in Theorem \ref{t:VB}. Furthermore, we obtain that $f_e \in \BMO_{\Delta}(\Rn)=\BMO(\Rn)$ and $\gamma_j(f_e;\Delta)=0$, $j=1,2,3$. Together with Proposition \ref{p:BMO}, the former gives that $f \in \BMO_e(\R^n_+)=\BMO_{\Delta_{N_+}}(\R^n_+)$ with $\|f\|_{\BMO_e(\R^n_+)} \simeq \|f\|_{\BMO_{\Delta_{N_+}}(\R^n_+)}$. Thus, to prove $f \in \VMO_{\Delta_{N_+}}(\R^n_+)$, it is enough to show $\gamma_j(f;\Delta_{N_+})=0$, $j=1,2,3$. For any $Q \subseteq \R^n_+$, it follows from \eqref{e:t-1} that
\begin{align*}
\bigg(\fint_{Q} |f-e^{-\ell(Q)^2 \Delta_{N_+}}f|^2 dx\bigg)^{\frac12}
=\bigg(\fint_{Q} |f_e -e^{-\ell(Q)^2 \Delta}f_e|^2 dx\bigg)^{\frac12}.
\end{align*}
Therefore, $\gamma_j(f_e;\Delta)=0$ implies that $\gamma_j(f;\Delta_{N_+})=0$, $j=1,2,3$.

Next let us demonstrate $\VMO_{\Delta_{N_+}}(\R^n_+) \subseteq \VMO_e(\R^n_+)$. Assume that $f \in \VMO_{\Delta_{N_+}}(\R^n_+)$, which is equivalent to $f \in \BMO_{\Delta_{N_+}}(\R^n_+)=\BMO_e(\R^n_+)$ and $\gamma_j(f;\Delta_{N_+})=0$, $j=1,2,3$. Thus we have $f_e \in \BMO(\Rn)=\BMO_{\Delta}(\Rn)$. We will show that $\gamma_j(f_e;\Delta)=0$, $j=1,2,3$.

Let $Q \subseteq \Rn$ be an arbitrary cube. If $Q \subseteq \R^n_+$, then the equation \eqref{e:t-1} gives that
\begin{align*}
\bigg(\fint_Q  |f_e-e^{-\ell(Q)^2 \Delta}f_e|^2 dx\bigg)^{\frac12}
=\bigg(\fint_Q  |f-e^{-\ell(Q)^2 \Delta_{N_+}}f|^2 dx\bigg)^{\frac12}.
\end{align*}

If $Q \subseteq \R^n_-$, then $x \in Q$ indicates that $\widetilde{x} \in \widetilde{Q} \subseteq \R^n_+$, where
\[
\widetilde{Q}=\{(x',x_n) \in \Rn;(x',-x_n) \in Q\}.
\]
It follows from \eqref{e:t-2} that
\begin{align*}
&\bigg(\fint_Q  |f_e(x)-e^{-\ell(Q)^2\Delta}f_e(x)|^2 dx\bigg)^{\frac12}
\\ %%%%%%%%%%%%%%%
&=\bigg(\fint_Q  |f(\widetilde{x})-e^{-\ell(Q)^2
\Delta_{N_+}}f(\widetilde{x}) |^2 dx\bigg)^{\frac12}
\\ %%%%%%%%%%%%%%%
&=\bigg(\fint_{\widetilde{Q}}
|f(x)-e^{-\ell(Q)^2 \Delta_{N_+}}f(x) |^2 dx\bigg)^{\frac12}.
\end{align*}

If $Q \cap \R^n_+ \neq \emptyset$ and $Q \cap \R^n_- \neq \emptyset$, then we still use the same notation as \eqref{e:Q+}.
Applying \eqref{e:t-1} and \eqref{e:t-2} again, we obtain that
\begin{align*}
&\bigg(\fint_Q  |f_e(x)-e^{-\ell(Q)^2\Delta}f_e(x)|^2 dx\bigg)^{\frac12}
\\ %%%%%%%%%%%%%%%
&\lesssim \bigg(\frac{1}{|\widehat{Q}_+|} \int_{\widehat{Q}_+}
|f_e(x)-e^{-\ell(Q)^2 \Delta}f_e(x)|^2 dx\bigg)^{\frac12}
\\ %%%%%%%%%%%%%%%
&\quad+ \bigg(\frac{1}{|\widehat{Q}_+|} \int_{\widehat{Q}_-}
|f_e(x)-e^{-\ell(Q)^2 \Delta}f_e(x)|^2 dx\bigg)^{\frac12}
\\ %%%%%%%%%%%%%%%
&= \bigg(\frac{1}{|\widehat{Q}_+|} \int_{\widehat{Q}_+}
|f(x)-e^{-\ell(Q)^2 \Delta_{N_+}}f(x)|^2 dx\bigg)^{\frac12}
\\ %%%%%%%%%%%%%%%
&\quad+ \bigg(\frac{1}{|\widehat{Q}_+|} \int_{\widehat{Q}_-}
|f(\widetilde{x})-e^{-\ell(Q)^2 \Delta_{N_+}}f(\widetilde{x})|^2 dx\bigg)^{\frac12}
\\ %%%%%%%%%%%%%%%
&= 2 \bigg(\frac{1}{|\widehat{Q}_+|} \int_{\widehat{Q}_+}
|f(x)-e^{-\ell(Q)^2 \Delta_{N_+}}f(x)|^2 dx\bigg)^{\frac12}.
\end{align*}
Consequently, $\gamma_j(f;\Delta_{N_+})=0$ implies that $\gamma_j(f_e;\Delta)=0$, $j=1,2,3$. This proves that $f_e \in \VMO_{\Delta}(\Rn)$. Additionally, it follows from Theorem \ref{t:VB} that $f_e \in \VMO(\Rn)$, which concludes that $f \in \VMO_e(\R^n_+)$.

\smallskip\noindent{\bf Step 2: $\VMO_e(\R^n_+)=\VMO_r(\R^n_+)$}. The first direction $\VMO_e(\R^n_+) \subseteq \VMO_r(\R^n_+)$ is easy. Indeed, for any $f \in \VMO_e(\R^n_+)$, there holds $f_e \in \VMO(\Rn)$. Thus, the fact $f_e|_{\R^n_+}=f$ gives that
\[
\|f\|_{\VMO_r(\R^n_+)} \leq \|f_e\|_{\VMO(\Rn)} = \|f\|_{\VMO_e(\R^n_+)}.
\]
It is enough to show that $\VMO_r(\R^n_+) \subseteq \VMO_e(\R^n_+)$.

Let $f \in \VMO_r(\R^n_+)$. For any $F \in \VMO(\Rn)$ with $F|_{\R^n_+}=f$, we have $F \in \BMO(\Rn)$ and $\gamma_j(F)=0$, $j=1,2,3$. Thus, in order to show $f \in \VMO_e(\R^n_+)$, it suffices to prove $f_e \in \BMO(\Rn)$ and $\gamma_j(f_e)=0$, $j=1,2,3$.

Let $1 \leq p < \infty$. Let $Q \subseteq \Rn$ be an arbitrary cube. If $Q \subseteq \R^n_+$, then there holds that
\begin{equation}\label{e:fe-1}
\bigg(\fint_Q  |f_e-f_{e,Q}|^p dx\bigg)^{\frac1p}
=\bigg(\fint_Q  |f-f_Q|^p dx\bigg)^{\frac1p}
=\bigg(\fint_Q  |F-F_Q|^p dx\bigg)^{\frac1p}. 
\end{equation}

If $Q \subseteq \R^n_-$, then $\widetilde{Q} \subseteq \R^n_+$ and
\begin{align*}
f_{e,Q}=\fint_Q  f(\widetilde{y}) dy =\fint_{\widetilde{Q}} f(y) dy =\fint_{\widetilde{Q}} F(y) dy.
\end{align*}
It immediately yields that
\begin{equation}\label{e:fe-2}
\begin{split}
&\bigg(\fint_Q  |f_e(x)-f_{e,Q}|^p dx\bigg)^{\frac1p}
=\bigg(\fint_Q  |f(\widetilde{x})-f_{\widetilde{Q}}|^p dx\bigg)^{\frac1p}
\\ %%%%%%%%%%%%%%%
&=\bigg(\fint_{\widetilde{Q}} |f(x)-f_{\widetilde{Q}}|^p dx\bigg)^{\frac1p}
=\bigg(\fint_{\widetilde{Q}} |F(x)-F_{\widetilde{Q}}|^p dx\bigg)^{\frac1p}.
\end{split}
\end{equation}

If $Q \cap \R^n_+ \neq \emptyset$ and $Q \cap \R^n_- \neq \emptyset$, then we by changing variables deduce that
\begin{align*}
&\bigg(\fint_Q  |f_e(x)-f_{e,Q}|^p dx\bigg)^{\frac1p}
\\ %%%%%%%%%%%%%%%
&\lesssim \bigg(\frac{1}{|Q|} \int_{\widehat{Q}_+} |f_e(x)-f_{e,Q}|^p dx\bigg)^{\frac1p}
+ \bigg(\frac{1}{|Q|} \int_{\widehat{Q}_-} |f_e(x)-f_{e,Q}|^p dx\bigg)^{\frac1p}
\\ %%%%%%%%%%%%%%%
&= \bigg(\frac{1}{|\widehat{Q}_+|} \int_{\widehat{Q}_+} |f(x)-f_{e,Q}|^p dx\bigg)^{\frac1p}
+ \bigg(\frac{1}{|\widehat{Q}_+|} \int_{\widehat{Q}_-} |f(\widetilde{x})-f_{e,Q}|^p dx\bigg)^{\frac1p}
\\ %%%%%%%%%%%%%%%
&= 2\bigg(\fint_{\widehat{Q}_+} |f(x)-f_{e,Q}|^p dx\bigg)^{\frac1p}.
\end{align*}
Using the same technique, we get
\begin{align*}
|f_{e,Q}-f_{\widehat{Q}_+}|
&\leq \fint_Q  |f_e(x)-f_{\widehat{Q}_+}| dx
\\ %%%%%%%%%%%%%%%
&\leq \frac{1}{|\widehat{Q}_+|} \int_{\widehat{Q}_+}|f(x)-f_{\widehat{Q}_+}| dx
+ \frac{1}{|\widehat{Q}_+|} \int_{\widehat{Q}_-} |f(\widetilde{x})-f_{\widehat{Q}_+}| dx
\\ %%%%%%%%%%%%%%%
&= 2 \fint_{\widehat{Q}_+} |f(x)-f_{\widehat{Q}_+}| dx
\leq 2\bigg(\fint_{\widehat{Q}_+} |f(x)-f_{\widehat{Q}_+}|^p dx\bigg)^{\frac1p}.
\end{align*}
Observe that
\begin{align*}
\bigg(\fint_{\widehat{Q}_+} |f(x)-f_{e,Q}|^p dx\bigg)^{\frac1p}
\leq |f_{e,Q}-f_{\widehat{Q}_+}| + \bigg(\fint_{\widehat{Q}_+} |f(x)-f_{\widehat{Q}_+}|^p dx\bigg)^{\frac1p}.
\end{align*}
Hence, putting them together, we obtain that
\begin{align}\label{e:fe-3}
\bigg(\fint_Q  |f_e-f_{e,Q}|^p dx\bigg)^{\frac1p}
\lesssim \bigg(\fint_{\widehat{Q}_+} |f-f_{\widehat{Q}_+}|^p dx\bigg)^{\frac1p}
=\bigg(\fint_{\widehat{Q}_+} |F-F_{\widehat{Q}_+}|^p dx\bigg)^{\frac1p}.
\end{align}

Accordingly, the inequalities \eqref{e:fe-1}, \eqref{e:fe-2} and \eqref{e:fe-3} for $p=1$ imply that $\|f_e\|_{\BMO(\Rn)} \lesssim \|F\|_{\BMO(\Rn)}$. Moreover, for $p=2$, we have $\gamma_j(f_e)=0$ provided by $\gamma_j(F)=0$, $j=1,2,3$. This completes the proof.
\end{proof}
%%%%%%%%%%%%%%%%%%%%%%%%%% End End End Proof %%%%%%%%%%%%%%%%%%%%%%%%%%

As a consequence, Theorem \ref{t:VMO} immediately follows from Theorems \ref{t:VMO-1} and \ref{t:VMO-2}.

We here give the comparison among the different spaces. 
%%%%%%%%%%%%%%%%%%%%%%%%%% Theorem Theorem Theorem %%%%%%%%%%%%%%%%%%%%%%
\begin{theorem}\label{t:VB}
The following inclusions hold
\begin{align*}
\VMO_{\Delta}(\Rn) = \VMO_{\sqrt{\Delta}}(\Rn) = \VMO(\Rn)
\subsetneqq \VMO_{\Delta_N}(\Rn) \subsetneqq \BMO_{\Delta_N}(\Rn).
\end{align*}
\end{theorem}
%%%%%%%%%%%%%%%%%%%%%%%%%% Theorem Theorem Theorem %%%%%%%%%%%%%%%%%%%%%%

%%%%%%%%%%%%%%%%%%%%%%%%%%%% Proof Proof Proof %%%%%%%%%%%%%%%%%%%%%%%%%
\begin{proof}
The equivalence $\VMO_{\Delta}(\Rn)=\VMO_{\sqrt{\Delta}}(\Rn)=\VMO(\Rn)$ was proved in \cite[~Proposition 3.6]{DDSTY}.

Next, we prove $\VMO(\Rn) \subsetneqq \VMO_{\Delta_N}(\Rn)$. For any $f \in \VMO(\Rn)$, it follows from Definition \ref{d:VMO} that $f_+ \in \VMO_r(\R^n_+)$ and $f_- \in \VMO_r(\R^n_-)$. Then Theorem \ref{t:VMO-1} gives that $f \in \VMO_{\Delta_N}(\Rn)$, which shows $\VMO(\Rn) \subseteq \VMO_{\Delta_N}(\Rn)$. In order to certify the strict inclusion, we give an example. Let $\varphi \in C_c^\infty(\R)$ be an even function with $\varphi(0)=1$, and $\psi \in C_c^\infty(\Rn)$ be
a radial function satisfying $\psi(x)=1$ for $|x|\le 1$. Set
\begin{align*}
g(x)& :=\varphi(x_n)\psi(x), \\
f(x)& :=g(x)\chi_{\{x_n>0\}}-g(x)\chi_{\{x_n<0\}}.
\end{align*}
Then, clearly $g\in \mathrm{VMO}(\Rn)$, and $f_{+,e}=g$, $f_{-,e}=-g$. So, we by Theorem \ref{t:VMO} get $f\in \mathrm{VMO}_{\Delta_N}(\Rn)$. However, $f\not\in \mathrm{VMO}(\mathbb R^n)$. Indeed, let $Q_\delta=[-\delta,\delta]^n$ for $\delta>0$. Then, for $\delta<1/{\sqrt n}$ it holds $f_{Q_\delta}=0$ and
\begin{equation*}
\frac{1}{|Q_\delta|}\int_{Q_\delta}|f(x)-f_{Q_\delta}|dx
=\frac{1}{|Q_\delta|}\int_{Q_\delta}|f(x)|dx
=\frac{1}{2\delta}\int_{-\delta}^{\delta}|\varphi(t)|dt\to 1,
\end{equation*}
as $\delta\to0$. This shows $f\not\in \mathrm{VMO}(\Rn)$.

Finally, let us show $\VMO_{\Delta_N}(\Rn) \subsetneqq \BMO_{\Delta_N}(\Rn)$. It suffices to present a function $f \in \BMO_{\Delta_N}(\Rn)$ such that $f \not\in \VMO_{\Delta_N}(\Rn)$, since the inclusion is obvious. Let $f(x)=\log|x|$. As is well known, $\log|x|\in\BMO(\mathbb R^n)$ but $\log|x|\not\in\VMO(\mathbb R^n)$. Then we see that $f_{+,e}=f$ and $f\not\in \VMO_{\Delta_N}(\Rn)$ by Theorem \ref{t:VMO}. It follows from Theorem 4.1 \cite{DDSY} that $\BMO(\Rn)\subsetneqq  \BMO_{\Delta_N}(\Rn)$. Therefore, $f$ is an example which belongs to $\BMO_{\Delta_N}(\Rn)$ but does not belong to $\VMO_{\Delta_N}(\Rn)$.

For the moment, we cannot find a reference which shows $\log|x|\not\in\mathrm{VMO}(\Rn)$. So, we will present its proof.
Let $B_\delta=\{x\in\R^n: |x|<\delta\}$. Let $f=-\log|x|$. Then, for $0<\delta<1$
\begin{align*}
f_{B_\delta}&=\frac{1}{|B_\delta|}\int_{B_\delta}\log\frac{1}{|x|}dx
=\frac{\omega_{n-1}}{v_n\delta^n}
\int_{0}^{\delta}\Bigl(\log\frac{1}{r}\Bigr)r^{n-1}dr
\\
&=\frac{n}{\delta^n}\biggl\{\Bigl[\frac{r^n}{n}\log\frac{1}{r}\Bigr]_0^\delta
+\frac1n\int_{0}^{\delta}r^{n-1}\cdot \frac{1}{r}dr\biggr\}
\\
&=\frac{n}{\delta^n}\bigg(\frac{\delta^n}{n}\log\frac1\delta +\delta^n\bigg)
=\log \frac1\delta +n.
\end{align*}
Observe that $\log\frac{1}{r}>\log \frac1\delta +n$ is equivalent to $0<r<\delta/e^n$. Then it yields that
\begin{align*}
\frac{1}{|B_\delta|}\int_{f(x)>f_{B_\delta}}|f(x)-f_{B_\delta}|dx
&=\frac{1}{|B_\delta|}\int_{{B_{\delta/e^n}}}(f(x)-f_{B_\delta})dx
\\
&=\frac{|B_{\delta/e^n}|}{|B_{\delta}|}
\frac{1}{|B_{\delta/e^n}|}\int_{{B_{\delta/e^n}}}(f(x)-f_{B_\delta})dx
\\
&=e^{-n}\biggl\{\Bigl(\frac1n \log\frac{e^n}{\delta}+n\Bigr) -\Bigl(\log \frac1\delta +n\Bigr)\biggr\}
=ne^{-n}.
\end{align*}
From this we get
\begin{equation*}
\frac{1}{|B_\delta|}\int_{B_\delta}|f(x)-f_{B_\delta}|dx=2ne^{-n}.
\end{equation*}
Thus, we deduce that
\begin{equation*}
\lim_{\delta\to0}\sup_{x\in\R^n,r\le\delta}
\frac{1}{|B(x,r)|}\int_{B(x,r)}|f(x)-f_{B(x,r)}|dx\ge 2ne^{-n}.
\end{equation*}
Hence, we see that $\log\frac{1}{|x|}\not\in\VMO(\Rn)$, which is equivalent to $\log|x|\not\in\VMO(\Rn)$.
\end{proof}
%%%%%%%%%%%%%% End End End Proof%%%%%%%%%%%%%%%%%%%%%

Let $X$ and $Y$ be Banach spaces. For the convenience of notation, we denote by $\overline{X}^Y$ the closure of $X$ in $Y$.
Now we characterize $\VMO$ spaces via smooth functions with compact supports.

%%%%%%%%%%%%%%%%%%%%%%% Theorem Theorem Theorem %%%%%%%%%%%%%%%%%%%%%%%%%
\begin{theorem}\label{t:CCC}
We have
\begin{align}
\label{e:VN} \VMO_{\Delta_N}(\Rn) &= \overline{C_c^{\infty}(\Rn)}^{\BMO_{\Delta_N}(\Rn)}, \\
\label{e:VN+} \VMO_{\Delta_{N_+}}(\R^n_+) &= \overline{C_c^{\infty}(\R^n_+)}^{\BMO_{\Delta_{N_+}}(\R^n_+)}, \\
\label{e:Ve} \VMO_e(\R^n_+) &= \overline{C_c^{\infty}(\R^n_+)}^{\BMO_e(\R^n_+)}, \\
\label{e:Vr} \VMO_r(\R^n_+) &= \overline{C_c^{\infty}(\R^n_+)}^{\BMO_r(\R^n_+)}, \\
\label{e:Vz} \VMO_z(\R^n_+) &= \overline{C_c^{\infty}(\R^n_+)}^{\BMO_z(\R^n_+)}.
\end{align}
Moreover, the similar results hold for the lower half-space.
\end{theorem}
%%%%%%%%%%%%%%%%%%%%%%% Theorem Theorem Theorem %%%%%%%%%%%%%%%%%%%%%%%%%

%%%%%%%%%%%%%%%%%%%%%%%%% Proof Proof Proof %%%%%%%%%%%%%%%%%%%%%%%%%%%%
\begin{proof}
It suffices to show \eqref{e:VN}, \eqref{e:Ve} and \eqref{e:Vz}, Since \eqref{e:VN+} and \eqref{e:Vr} follow from \eqref{e:Ve}, Theorem \ref{t:VMO-2} and Proposition \ref{p:BMO}. We begin with the proof of \eqref{e:VN}.

We will use the fact
\begin{equation}\label{e:C}
\VMO(\Rn) = \overline{C_c(\Rn)}^{\BMO(\Rn)}.
\end{equation}
Here $C_c(\Rn)$ is the collection of continuous functions with compact support. Indeed, $\VMO(\Rn) \subseteq \overline{C_c(\Rn)}^{\BMO(\Rn)}$ follows from $\VMO(\Rn) = \overline{C_c^{\infty}(\Rn)}^{\BMO(\Rn)}$. Now let 
$f \in \overline{C_c(\Rn)}^{\BMO(\Rn)}$. Then for any $\epsilon>0$, there exists $g \in C_c(\Rn)$ such that $\|f-g\|_{\BMO(\Rn)} < \epsilon$. As a consequence, it yields that $g \in L^1(\Rn) \cap L^{\infty}(\Rn)$, and $g$ is uniformly continuous on $\Rn$. For any cube $Q \subseteq \Rn$, we have
\begin{align*}
\MO(f,Q) &\leq \MO(f-g,Q) + \MO(g,Q)
\\%%%%%%%%%%
&\leq \|f-g\|_{\BMO(\Rn)} + \MO(g,Q)
< \epsilon + \MO(g,Q),
\end{align*}
where $\MO(f,Q)=\fint_Q |f(y)-f_Q|dy$. Thus, $f \in \BMO(\Rn)$ with $\|f\|_{\BMO(\Rn)} \lesssim \epsilon+\|g\|_{L^{\infty}(\Rn)}$. Additionally, $\gamma_1(f)=0$ since $g$ is uniformly continuous, and
\begin{align*}
\MO(f,Q) < \epsilon + \fint_{Q}\fint_{Q} |g(x)-g(y)| dx dy.
\end{align*}
$\gamma_2(f)=0$ follows from $g \in L^1(\Rn)$, and $\MO(f,Q) < \epsilon + \frac{2}{|Q|} \|g\|_{L^1(\Rn)}$. For any cube $Q$ with $Q \cap \supp(g)=\emptyset$, there holds $\MO(f,Q) < \epsilon + \MO(g,Q) = \epsilon$, which implies that $\gamma_3(f)=0$. This shows $f \in \VMO(\Rn)$.

Let us continue our proof. We first prove $\overline{C_c^{\infty}(\Rn)}^{\BMO_{\Delta_N}(\Rn)} \subseteq \VMO_{\Delta_N}(\Rn)$. Assume that $f \in \overline{C_c^{\infty}(\Rn)}^{\BMO_{\Delta_N}(\Rn)}$. Then for any $\epsilon>0$, there exists $g \in C_c^{\infty}(\Rn)$ such that $\|f-g\|_{\BMO_{\Delta_N}(\Rn)} < \epsilon$. Observe that $g_{+,e},g_{-,e} \in C_c(\Rn)$.
Indeed, if $\supp(g) \subseteq \R^n_+$, then $g_{+,e} \in C_c(\Rn)$ and $g_{-,e} \equiv 0$. If $\supp(g) \subseteq \R^n_-$, then $g_{+,e} \equiv 0$ and $g_{-,e} \in C_c(\Rn)$. If $\supp(g) \cap \R^n_+ \neq \emptyset$ and $\supp(g) \cap \R^n_- \neq \emptyset$, then $g_{+,e} \in C_c(\Rn)$ and $g_{-,e} \in C_c(\Rn)$. Moreover, it follows from Proposition \ref{p:BMO} that $\|f_{+,e}-g_{+,e}\|_{\BMO(\Rn)} \lesssim \epsilon$ and $\|f_{-,e}-g_{-,e}\|_{\BMO(\Rn)} \lesssim \epsilon$. By \eqref{e:C}, we have $f_{+,e} \in \VMO(\Rn)$ and $f_{-,e} \in \VMO(\Rn)$, which together with Theorem \ref{t:VMO} gives $f \in \VMO_{\Delta_N}(\Rn)$.

Now we are in the position to show the converse. Assume that $f \in \VMO_{\Delta_N}(\Rn)$, which by Theorem \ref{t:VMO} gives that $f_{+,e}, f_{-,e} \in \VMO(\Rn)$. Hence for any $\epsilon>0$, there exists $\tilde g_1\in C_c^\infty(\Rn)$ such that
$\|f_{+,e}-\tilde g_1\|_{\BMO(\Rn)}<\epsilon $. Set
$$
g_1(x)=(\tilde g_1(x)+\tilde g_1(x',-x_n))/2, \ \ x=(x',x_n) \in \Rn.
$$
Then we see that $g_1(x)=g_1(x',-x_n)=:g_1(\widetilde{x})$ and
\begin{align*}
\|f_{+,e}-g_1\|_{BMO(\Rn)}
&=\|f_{+,e}-(\tilde g_1(x)+\tilde g_1(x',-x_n))/2\|_{\BMO(\Rn)}
\\%%%%%%%%%%
&\le \frac12\big(\|f_{+,e}-\tilde g_1\|_{\BMO(\Rn)}
+\|f_{+,e}-\tilde g_1(x',-x_n)\|_{\BMO(\Rn)}\big)
\\%%%%%%%%%%
&=\|f_{+,e}-\tilde g_1\|_{\BMO(\Rn)}
<\epsilon.
\end{align*}
Similarly, there exist $g_2\in C_c^\infty(\Rn)$ such that
\begin{equation*}
g_2(x)=g_2(\widetilde{x})  \ \text{ and } \ \|f_{-,e}-g_2\|_{\BMO(\Rn)}<\epsilon .
\end{equation*}
By Lemma \ref{lem:VMO-3} below there exist even functions $\psi_1,\,\psi_2\in C_c^\infty(\R)$ such that $\psi_1(0)=\psi_2(0)=1$ and
\begin{equation*}
\|g_1(x',0)\psi_1(x_n)\|_{\BMO(\Rn)} + \|g_2(x',0)\psi_2(x_n)\|_{\BMO(\Rn)} < \varepsilon .
\end{equation*}
Define
\begin{equation*}
h(x)=\begin{cases}
g_1(x)+g_2(x',0)\psi_2(x_n), &x \in \R^n_+, \\
g_1(x',0)\psi_1(x_n)+g_2(x), &x \in \R^n_-.
\end{cases}
\end{equation*}
It immediately yields that
\begin{equation*}
h_{+,e}(x)= g_1(x)+g_2(x',0)\psi_2(x_n), \quad h_{-,e}(x)= g_1(x',0)\psi_1(x_n) +g_2(x),
\end{equation*}
$h \in C_c(\Rn)$ and $h_{+,e}, h_{-,e} \in C_c^{\infty}(\Rn) \subseteq \VMO(\Rn)$. Consequently, we deduce that
\begin{align*}
\|f_{+,e}-h_{+,e}\|_{\BMO(\Rn)}
&\le \|f_{+,e}-g_1\|_{\BMO(\Rn)}+\|g_1-h_{+,e}\|_{\BMO(\Rn)}
\\ %%%%%%%%%%
&\le \epsilon +\|g_2(x',0)\psi_2(x_n)\|_{BMO(\Rn)}<2\epsilon,
\\ %%%%%%%%%%
\|f_{-,e}-h_{-,e}\|_{\BMO(\Rn)}
&\le \|f_{-,e}-g_2\|_{\BMO(\Rn)}+\|g_2-h_{-,e}\|_{\BMO(\Rn)}
\\ %%%%%%%%%%
&\le \epsilon +\|g_1(x',0)\psi_1(x_n)\|_{\BMO(\Rn)}<2\epsilon,
\end{align*}
and
\begin{align*}
\|f-h\|_{\VMO_{\Delta_N}(\Rn)}
\simeq  \|f_{+,e}-h_{+,e}\|_{\VMO(\Rn)} + \|f_{-,e}-h_{-,e}\|_{\VMO(\Rn)}
\lesssim \epsilon.
\end{align*}
This implies that $C_c(\Rn)$ is dense in $\VMO_{\Delta_N}(\Rn)$. Since $C_c^\infty(\Rn)$ is dense in $C_c(\Rn)$ under the $L^{\infty}(\Rn)$ norm, we see that $C_c^\infty(\Rn)$ is dense in $\VMO_{\Delta_N}(\Rn)$.

Now we present the proof of \eqref{e:Ve}. Let $f\in \VMO_{e}(\R_+^n)$. Then $f_{e}\in \VMO(\Rn)$, and hence for fixed $\epsilon >0$ there exists $g\in C_c^\infty(\Rn)$ such that $g(x)=g(x',-x_n)$ and $\|f_{e}-g\|_{\BMO(\Rn)}<\epsilon$. By Lemma \ref{lem:VMO-3} there exist an even function $\psi\in C_c^\infty(\R)$ such that
\begin{equation*}
\psi(0)=1 \ \text{ and } \
\|g(x',0)\psi(x_n)\|_{\BMO(\Rn)}< \epsilon .
\end{equation*}
Define
\begin{equation*}
h(x)=f_e(x)-g(x',0)\psi(x_n), \ x \in \Rn.
\end{equation*}
Then we have $h \in C_c^{\infty}(\Rn)$ and
\begin{equation*}
\|f_{e}-h\|_{\BMO(\Rn)}<\epsilon\ \text{ and }h(x',0)=0.
\end{equation*}
Approximating $h$ uniformly, we can find $\tilde h\in C_c^\infty(\Rn)$ such that $\tilde h(x',-x_n)=\tilde h(x)$ and for some $\delta>0$
\begin{equation*}
\|f_{e}-\tilde h\|_{\BMO(\Rn)}<2\epsilon\ \text{ and }\tilde h(x',x_n)=0
\ \text{ for } -\delta<x_n<\delta.
\end{equation*}
Then $\tilde{h}_+ \in C_c^\infty(\R_+^n)$ and
\begin{equation*}
\|f-\tilde{h}_+\|_{\BMO_{e}(\R_+^n)}
=  \|f_{e}-\tilde{h}\|_{\BMO(\Rn)} <2\epsilon.
\end{equation*}
This yields $\VMO_e(\R_+^n)
\subseteq \overline{C_c^\infty(\R_+^n)}^{\BMO_e(\Rn)}$.
The converse inclusion relation is trivial.

The proof of \eqref{e:Vz} is similar to \eqref{e:Ve}. 
\end{proof}
%%%%%%%%%%%%%%%%%%%%%%%%% End End End Proof %%%%%%%%%%%%%%%%%%%%%%%%%%%

The remainder of this section is devoted to showing Lemma \ref{lem:VMO-3}. To this end, we first present the $\BMO$ estimates for smooth functions with compact supports.

%%%%%%%%%%%%%%%%%%%%%%%% Lemma Lemma Lemma VMO-1 %%%%%%%%%%%%%%%%%%%%%%
\begin{lemma}\label{lem:VMO-1}
Let $\varphi_j(x)\in C_c^\infty(\R)$ be nonnegative and satisfy $\chi_{\{|x|\le 2^j-1\}}\le \varphi_j(x) \le \chi_{\{|x|\le 2^j\}}$, $j \in \N$, and
\begin{equation*}
\psi_\ell(x) = \frac{1}{\ell} \sum_{j=1}^{\ell}\varphi_j(x),\ \ \ell=3,4,\dots.
\end{equation*}
Then it holds that $\psi_\ell\in C_c^\infty(\R)$, $0\le \psi_\ell(x)\le 1$, $\psi_\ell(x)=1$ for $|x|\le 1$, $\psi_\ell(x)=0$ for $|x|\ge 2^\ell$ and $\|\psi_\ell\|_{\BMO(\R)} \le 16/\ell$,
\end{lemma}
%%%%%%%%%%%%%%%%%%%%%%%% Lemma Lemma Lemma VMO-1 %%%%%%%%%%%%%%%%%%%%%%

%%%%%%%%%%%%%%%%%%%%%%%%% Proof of Lemma VMO-1 %%%%%%%%%%%%%%%%%%%%%%%%%
\begin{proof}
We have only to prove $\|\ell\,\psi_\ell\|_{\BMO(\R)}\le 16$. Let $I=[a,b]$ be an interval. If $|I| < 2$, then we see easily that $(2^{j+1}-1)-(2^j-1)=2^j \geq 2$ for any $j \geq 1$ and 
\begin{equation*}
\fint_I |\ell\psi_\ell(x)-(\ell\psi_\ell)_I| dx
\leq \sum_{j=1}^{\ell} \fint_I |\varphi_j(x)-(\varphi_j)_I| dx 
\le 2.
\end{equation*}
If $|I|\ge 2$, there exists $j\in\N_+$ such that $2^j\le |I|< 2^{j+1}$. We shall consider the following two cases: (a) $j\ge \ell-1$ and (b) $1\le j\le \ell-2$.

\par \smallskip\noindent{\bf Case (a): $j\ge \ell-1$.}
We get
\begin{align*}
\fint_I|\ell\psi_{\ell}(x)|dx
&\le \sum_{i=1}^{\ell-2} \fint_{I} \varphi_i(x)dx
+\sum_{i=ell-1}^{\ell}\fint_{I}\varphi_i(x)dx
\\
&\le \sum_{i=1}^{\ell-2}\frac{1}{2^j}\dot 2\cdot 2^i +2
=\frac{2^{\ell-1}}{2^j}+2<3.
\end{align*}

\par \smallskip\noindent{\bf Case (b):  $1\le j\le \ell-2$.}
If $a \ge 2^{j}$, then $I$ contains at most one point of the form $2^i$ $(j+1\le i\le \ell)$. Hence, it yields that
\begin{equation*}
\fint_I|\ell\psi_{\ell}(x)-(\ell-i)|dx\le 1.
\end{equation*}
If $-2^j\le a < 2^{j}$, then $0\le b< 2^{j+1}+2^{j-1}<2^{j+2}\le 2^\ell$. 
Hence we get
\begin{equation*}
\fint_I|\ell\psi_{\ell}(x)-(\ell-j-1)|dx
\le \frac{1}{2^{j}}\sum_{i=1}^{j+1}2\cdot 2^i\le \frac{1}{2^j}\cdot 2\cdot 2^{j+2}=8.
\end{equation*}
Similarly, the above estimates hold for $b<2^j$. Noting that $a<2^j$ implies $b<2^{j+1}-2^j=2^j$,  we get
\begin{equation*}
\fint_I|\ell\psi_\ell(x)-(\ell\psi_\ell)_I|dx\le 16.
\end{equation*}
This completes the proof. 
\end{proof}
%%%%%%%%%%%%%%%%%%%%%%% End Proof of Lemma VMO-1 %%%%%%%%%%%%%%%%%%%%%%%%

%%%%%%%%%%%%%%%%%%%%%% Lemma Lemma Lemma VMO-2 %%%%%%%%%%%%%%%%%%%%%%%
\begin{lemma}\label{lem:VMO-2}
For any $\epsilon,\,\eta >0$, there exists
$\psi\in C_c^\infty(\R)$ with $\supp \psi\subset [-\eta, \eta]$ such that
\begin{equation*}
\|\psi\|_{\BMO(\R)}<\epsilon,\
0\le \psi(x)\le 1, \
\text{ and }\ \psi(0)=1.
\end{equation*}
\end{lemma}
%%%%%%%%%%%%%%%%%%%%%%%% Lemma Lemma Lemma VMO-2 %%%%%%%%%%%%%%%%%%%%%%

%%%%%%%%%%%%%%%%%%%%%%%% Proof of Lemma VMO-2 %%%%%%%%%%%%%%%%%%%%%%%%%%
\begin{proof}
We use the notations in Lemma \ref{lem:VMO-1}. First we take $\ell\in\N_+$ so that $\|\psi_\ell\|_{\BMO(\R)}\le 16/\ell<\epsilon $, and we set
\begin{equation*}
\psi(x)=\psi_\ell(2^\ell x/\eta).
\end{equation*}
Then from the dilation invariance of $\BMO$ norm, we get $\|\psi\|_{\BMO(\R)}<\epsilon$. Since the $\supp \psi_\ell\subset [-2^\ell,\,2^\ell]$, we see that $\supp \psi\subset  [-\eta,\,\eta]$. This $\psi$ also satisfies $\psi(0)=1$ and $0\le \psi(x)\le 1$ for $x\in\R$.
\end{proof}
%%%%%%%%%%%%%%%%%%%%%%%% end Proof of Lemma VMO-2 %%%%%%%%%%%%%%%%%%%%%%%%

%%%%%%%%%%%%%%%%%%%%%%%%% Lemma Lemma Lemma VMO-3 %%%%%%%%%%%%%%%%%%%%%
\begin{lemma}\label{lem:VMO-3}
For any $\epsilon >0$ and $g\in C_c^\infty(\Rn)$, there exists
$\psi\in C_c^\infty(\R)$ such that
\begin{equation*}
0\le \psi(x_n)\le 1,\ \psi(0)=1
\text{ and }\|g(x',0)\psi(x_n)\|_{\BMO(\Rn)}<\epsilon,
\end{equation*}
where $x=(x',x_n) \in \Rn$.
\end{lemma}
%%%%%%%%%%%%%%%%%%%%%%%%% Lemma Lemma Lemma VMO-3 %%%%%%%%%%%%%%%%%%%%%

%%%%%%%%%%%%%%%%%%%%%%%%%%% Proof of Lemma VMO-3 %%%%%%%%%%%%%%%%%%%%%%%
\begin{proof} 
In the case $n=1$, for $\epsilon>0$ take $\epsilon_1>0$ satisfying  $|g(0)|\epsilon _1<\epsilon$. Takeing $\psi_\ell$ in Lemma 
\ref{lem:VMO-1} so that $16/\ell<\epsilon_1$, we get  $\|g(0)\psi_\ell(x_n)\|_{\BMO(\R)}<\epsilon$. So this $\psi_\ell$ is a desired function. In the case $n\ge2$, we proceed as follows.  Let $g\in C_c^\infty(\Rn)$ and $\epsilon >0$. Then $g(x',0)\in C_c^\infty(\R^{n-1})\subset \VMO(\R^{n-1})$. Let $Q=(I',I)$ be any cube in $\Rn$, where $I'$ is a cube in $\R^{n-1}$ and $I$ be an interval in $\R$. Since $g(x',0)$ is also a $\VMO(\R^{n-1})$ function, there exists  $\delta>0$ such that
\begin{equation*}
\frac{1}{|I'|}\int_{I'}|g(x',0)-g(\cdot,0)_{I'}|dx'<\epsilon \ \text{ if }
|I'|<\delta^{n-1}.
\end{equation*}
By Lemma \ref{lem:VMO-2} for $\eta=\epsilon \delta/2$,  there exists $\psi\in C_c^\infty(\R)$ with $\supp \psi\subset [-\eta, \eta]$ such that
\begin{equation*}
\|\psi\|_{\BMO(\R)}<\epsilon,\ \psi(0)=1, \text{ and } 0\le \psi(x)\le 1,\ x\in\R.
\end{equation*}
Now we deduce that
\begin{align*}
J&:=\frac{1}{|Q|}\int_{I'}\int_I|g(x',0)\psi(x_n)-g(\cdot,0)_{I'}\psi_I|dx_ndx'
\\%%%%%%%%%%%%  1 %%%%%%%%%%%%%
&\le \frac{1}{|Q|}\int_{I'}\int_I|g(x',0)\psi(x_n)-g(x',0)\psi_I|dx_ndx'
\\%%%%%%%%%%%%%%%%%
&\qquad+\frac{1}{|Q|}\int_{I'} \int_I
|g(x',0)\psi_I-g(\cdot,0)_{I'}\psi_I|dx_ndx'
\\%%%%%%%%%%%%%%%%%
&\le \frac{1}{|I'|}\int_{I'}|g(x',0)|dx'
\frac{1}{|I|}\int_I|\psi(x_n)-\psi_I|dx_n
\\%%%%%%%%%%%%%%%%%
&\qquad +\frac{1}{|I'|}\int_{I'}|g(x',0)-g(\cdot,0)_{I'}|dx'|\psi_I|
\\%%%%%%%%%%%%%%%%%
&< \epsilon  \|g\|_{L^\infty(\Rn)}
+\frac{1}{|I'|}\int_{I'}|g(x',0)-g(\cdot,0)_{I'}|dx'|\psi_I|.
\end{align*}
Hence, if $|I|=|I'|^{1/(n-1)}<\delta$, we get $J<\epsilon \|g\|_{L^\infty(\Rn)}+\epsilon$. If $I\cap [-\eta,\eta]=\emptyset$, we see trivially $J=0$. If $I\cap [-\eta,\eta]\ne\emptyset$ and $|I|\ge \delta$, we see that
\begin{equation*}
|\psi_I|\le \frac{1}{|I|}\int_I|\psi(x_n)|dx_n\le \frac{1}{|I|}\cdot 2\eta
<\frac{2\eta}{\delta}<\epsilon,
\end{equation*}
and so we get $J< \epsilon \|g\|_{L^\infty(\Rn)} +2\epsilon \|g\|_{L^\infty(\Rn)}=3\epsilon \|g\|_{L^\infty(\Rn)}$. Modifying constants above completes the proof of Lemma \ref{lem:VMO-3}.
\end{proof}
%%%%%%%%%%%%%%%%%%%%%%% end Proof of Lemma VMO-3 %%%%%%%%%%%%%%%%%%%%%%%%

%%%%%%%%%%%%%%%%%%%%%% SECTION  SECTION SECTION %%%%%%%%%%%%%%%%%%%%%%%
%%%%%%%%%%%%%%%%%%%%%% SECTION  SECTION SECTION %%%%%%%%%%%%%%%%%%%%%%%
\section{Dual spaces}\label{s:dual}

Let us recall the definitions of various Hardy spaces on the upper/lower half-space in \cite{CKS}.
%%%%%%%%%%%%%%%%%%%%% DEFINITION DEFINITION DEFINITION %%%%%%%%%%%%%%%%%%%%
\begin{definition}
Let $f$ be a function on $\R^n_+$.
\begin{enumerate}
\item[(1)]  $f$ is said to be in $H^1_r(\R^n_+)$ if there exists
$F\in H^1(\Rn)$ such that $F|_{\R^n_+} = f$. If $f \in H^1_r(\R^n_+)$, we set
$\|f\|_{H^1_r(\R^n_+)} :=\inf \big\{\|F\|_{H^1(\Rn)}: F|_{\R^n_+} = f\big\}$.
\item[(2)] $f$ is said to be in $H^1_z(\R^n_+)$ if the function $f_z$ belongs to
$H^1(\Rn)$. If $f \in H^1_z(\R^n_+)$, we set
$\|f\|_{H^1_z(\R^n_+)} := \|f_z\|_{H^1(\Rn)}$.
\item[(3)] $f$ is said to be $H^1_e(\R^n_+)$ if $f_e \in H^1(\Rn)$.
Moreover, $H^1_e(\R^n_+)$ is endowed with the norm
$\|f\|_{H^1_e(\R^n_+)} := \|f_e\|_{H^1(\Rn)}$.
\end{enumerate}

Similarly one can define the spaces $H^1_r(\R^n_-)$, $H^1_z(\R^n_-)$ and $H^1_e(\R^n_-)$.
\end{definition}
%%%%%%%%%%%%%%%%%%%%% DEFINITION DEFINITION DEFINITION %%%%%%%%%%%%%%%%%%%%

The authors in \cite{CKS} proved that 
\begin{align*}
H^1_r(\R^n_+) = H^1_o(\R^n_+) 
\ \ \text{ and } \ \ 
H^1_z(\R^n_+)=H^1_e(\R^n_+). 
\end{align*} 
A celebrated work of Fefferman and Stein \cite{FS} showed that $\BMO(\Rn)$ is the dual space of $H^1(\Rn)$. Moreover, in the half-spaces setting, the duality was established in \cite{ART} as follows
\begin{align*}
(H^1_r(\R^n_+))^* = \BMO_z(\R^n_+) 
\ \ \text{ and } \ \ 
(H^1_z(\R^n_+))^* = \BMO_r(\R^n_+). 
\end{align*}

%%%%%%%%%%%%%%%%%%%%%% THEOREM THEOREM THEOREM %%%%%%%%%%%%%%%%%%%%%%
\begin{theorem}\label{t:VMO-H}
The dual space of $\VMO_{\Delta_N}(\Rn)$ is $H^1_{\Delta_N}(\Rn)$.
\end{theorem}
%%%%%%%%%%%%%%%%%%%%%% THEOREM THEOREM THEOREM %%%%%%%%%%%%%%%%%%%%%%

%%%%%%%%%%%%%%%%%%%%%%%%% Proof Proof Proof %%%%%%%%%%%%%%%%%%%%%%%%%%%%
\begin{proof}
The proof can be found in Theorem 4.1 \cite{DDSTY}, in which a more general result about the operator $L$ was given.
\end{proof}
%%%%%%%%%%%%%%%%%%%%%%%%% End End End Proof  %%%%%%%%%%%%%%%%%%%%%%%%%%%

Based on the duality above, let us investigate the weak*-convergence in $H^1_{\Delta_N}(\Rn)$.
%%%%%%%%%%%%%%%%%%%%%%% THEOREM THEOREM THEOREM %%%%%%%%%%%%%%%%%%%%%
\begin{theorem}\label{t:weak}
Suppose that $\{f_k\}_{k \geq 1}$ is a bounded sequence in $H^1_{\Delta_N}(\Rn)$, and that $\lim\limits_{k \to \infty} f_k(x)=f(x)$ a.e. $x \in \Rn$. Then $f \in H^1_{\Delta_N}(\Rn)$ and $\{f_k\}_{k \geq 1}$ weak*-converges to $f$, that is,
\begin{align*}
\lim_{k \to \infty} \int_{\Rn} f_k(x) \phi(x) dx
=\int_{\Rn} f(x) \phi(x) dx, \ \ \forall \phi \in \VMO_{\Delta_N}(\Rn), 
\end{align*}
where the integrals denote the dual form between $H^1_{\Delta_N}(\R^n)$ and $\BMO_{\Delta_N}(\R^n)$ in general.
\end{theorem}
%%%%%%%%%%%%%%%%%%%%%% THEOREM THEOREM THEOREM %%%%%%%%%%%%%%%%%%%%%%

%%%%%%%%%%%%%%%%%%%%%%%% Proof Proof Proof %%%%%%%%%%%%%%%%%%%%%%%%%%%%%
\begin{proof}
Without loss of generality, we may assume that $\sup\limits_{k \geq 1}\|f_k\|_{H^1_{\Delta_N}(\Rn)} \leq 1$. The Banach-Alaoglu theorem states that if $X$ a Banach space, then the closed unit ball in $X^*$ is compact in the weak* topology. Together with Theorem \ref{t:VMO-H}, it yields that there exists a subsequence $\{f_{k_j}\}$ of $\{f_k\}$ such that $\{f_{k_j}\}$ weak*-converges to $g$ for some $g \in H^1_{\Delta_N}(\Rn)$. Thus, we have
\begin{equation*}
\lim_{j \to \infty} \int_{\Rn} f_{k_j}(x) \phi(x) dx
=\int_{\Rn} g(x) \phi(x) dx, \ \ \forall \phi \in C_{c}^\infty(\Rn).
\end{equation*}
Since $\|f_k\|_{H^1(\Rn)}\lesssim \|f_k\|_{H^1_{\Delta_N}(\Rn)}\le 1$, Jones and Journ\'e's theorem \cite{JJ} implies that $f_{k_j}$ weak*-converges to $f$ in $H^1(\Rn)$. So, there holds
\begin{equation*}
\lim_{j \to \infty} \int_{\Rn} f_{k_j}(x) \phi(x) dx
=\int_{\Rn} f(x) \phi(x) dx, \ \ \forall \phi \in C_{c}^\infty(\Rn).
\end{equation*}
And hence
\begin{equation*}
\int_{\Rn} f(x) \phi(x) dx
=\int_{\Rn} g(x) \phi(x) dx, \ \ \forall \phi \in C_{c}^\infty(\Rn).
\end{equation*}
That is, $f=g$ as distributions. Since both are in $L_{\rm loc}^1(\Rn)$, we have $f=g$ a.e in $\Rn$, and it follows $f \in H^1_{\Delta_N}(\Rn)$.

Let $\phi \in \VMO_{\Delta_N}(\Rn)$. It follows from Theorem \ref{t:CCC} that for any $\epsilon>0$, there exists $\phi_{\epsilon} \in C_c^{\infty}(\Rn)$ such that $\|\phi-\phi_{\epsilon}\|_{\BMO_{\Delta_N}(\Rn)} < \epsilon$. By Theorem 3.13 in \cite{LW}, we have that $H^1_{\Delta_N}(\Rn) \subsetneq H^1(\Rn)$. Then using weak*-convergence in $H^1(\Rn)$ given in \cite{JJ}, we obtain
\begin{align*}
\lim_{k \to \infty} \int_{\Rn} f_k(x) \phi_{\epsilon}(x) dx
=\int_{\Rn} f(x) \phi_{\epsilon}(x) dx.
\end{align*}
Consequently, there exists $K \in \N_+$ such that for any $k \geq K$
\begin{align*}
\bigg|\int_{\Rn} (f_k-f) \phi dx \bigg|
& \leq \bigg|\int_{\Rn} (f_k-f) \phi_{\epsilon} dx\bigg|
+ \bigg|\int_{\Rn} (f_k-f) (\phi-\phi_{\epsilon}) dx\bigg|
\\ %%%%%%%%%%
&\leq \epsilon + \big(\|f_k\|_{H^1_{\Delta_N}(\Rn)} + \|f\|_{H^1_{\Delta_N}(\Rn)}\big)
\|\phi-\phi_{\epsilon}\|_{\BMO_{\Delta_N}(\Rn)}
\\ %%%%%%%%%%
&\leq \epsilon + \big(1 + \|f\|_{H^1_{\Delta_N}(\Rn)}\big) \epsilon
\lesssim  \epsilon.
\end{align*}
The proof is complete. 
\end{proof}
%%%%%%%%%%%%%%%%%%%%%%%%%%% End End Proof  %%%%%%%%%%%%%%%%%%%%%%%%%%%

%%%%%%%%%%%%%%%%%%%%%%%% THEOREM THEOREM THEOREM %%%%%%%%%%%%%%%%%%%%
\begin{theorem}\label{t:VMOe-H}
The dual space of $\VMO_z(\R^n_+)$ is $H^1_r(\R^n_+)$.
\end{theorem}
%%%%%%%%%%%%%%%%%%%%%%%% THEOREM THEOREM THEOREM %%%%%%%%%%%%%%%%%%%%

%%%%%%%%%%%%%%%%%%%%%%%%%% Proof Proof Proof %%%%%%%%%%%%%%%%%%%%%%%%%%
\begin{proof}
We will use a functional analysis method. Let $\mathfrak{B}$ be a Banach space and $\mathfrak{B}^*$ be the dual space. Suppose that $\mathfrak{B}_1$ is a closed subspace of  $\mathfrak{B}$. Denote 
\[
\mathfrak{B}_1^{\perp} = \big\{\mathscr{L} \in \mathfrak{B}^*: \mathscr{L}|_{\mathfrak{B}_1} = 0\big\}.
\]
Then $\mathfrak{B}_1^{\perp}$ is a closed subspace of $\mathfrak{B}^*$. It was given in \cite[p. 89]{RS} that
\begin{equation*}
\mathfrak{B}_{1}^{*} \cong \mathfrak{B}^{*} /\mathfrak{B}_{1}^{\perp}.
\end{equation*}
Taking 
\begin{align*}
\mathfrak{B}=\VMO(\Rn)  \ \text{ and } \  \mathfrak{B}_{1}=\VMO_z(\R^n_+),
\end{align*} 
we have
\begin{equation*}
\mathfrak{B}_{1}^{\perp}=\big\{\mathscr{L} \in H^1(\Rn) :
\mathscr{L}(f)=0, \forall f \in \VMO_z(\R^n_+)\big\}
\end{equation*}
provided by $(\VMO(\Rn))^*=H^1(\Rn)$. Consequently, for any $\mathscr{L} \in \mathfrak{B}_{1}^{\perp}$, there exists $g \in H^1(\Rn)$ such that
\begin{equation}\label{e:Lf}
\mathscr{L}(f)=\int_{\Rn} f(x) g(x) dx
=\int_{\R^n_+} f(x) g(x) dx=0,\ \ \forall f \in \VMO_z(\R^n_+), 
\end{equation}
where the integrals denote the dual form between $H^1_{r}(\R^n_+)$ and $\BMO_{z}(\R^n_+)$ in general.

Let $Q$ be a cube in $\R^n_+$. For $Q$, there exists a sequence $\{\varphi_j\}_j\subset C_c^\infty(\Rn)$ with support in $\R_+^n$ such that 
\begin{equation*}
0 \le \varphi_1(x) \le  \varphi_2(x) \le \cdots \le \mathbf{1}_Q(x), \ \ \ 
\lim_{j \to \infty} \varphi_j(x) = \mathbf{1}_Q(x), x \in \Rn. 
\end{equation*}
Then, $\varphi_j\in  \VMO_z(\R_+^n)$ and so by \eqref{e:Lf} we get  $\int_\Rn\varphi_j(x)g(x)dx=0$, $j=1,2,\ldots$.  However, since $g \in H^1(\Rn) \subsetneq L^1(\Rn)$, it follows from Lebesgue's dominated convergence theorem that 
\begin{equation*}
0=\lim_{j\to\infty}\int_\Rn\varphi_j(x)g(x)dx=\int_Q g(x)dx.
\end{equation*}
Thus, we obtain by Lebesgue's differentiation theorem that 
\begin{equation*}
g(x)=0\ \text{ a.e. }x\in \R_+^n. 
\end{equation*}
It indicates that 
\begin{equation*}
\mathfrak{B}_{1}^{\perp} \cong \{g \in H^1(\Rn) : g \equiv 0\text { on } \R^n_+\}.
\end{equation*}
By definition, it yield that
\begin{equation*}
\mathfrak{B}^{*} / \mathfrak{B}_{1}^{\perp} \cong H^1(\Rn)/\{g \in H^1(\Rn): g \equiv 0 \text{ on } \R^n_+ \}
=H^1_{r}(\R^n_+).
\end{equation*}
This shows that $H^1_r(\R^n_+)$ is the dual space of $\VMO_z(\R^n_+)$. 
\end{proof}
%%%%%%%%%%%%%%%%%%%%%%%%% End End End Proof  %%%%%%%%%%%%%%%%%%%%%%%%%%%

%%%%%%%%%%%%%%%%%%%%%%%%%% THEOREM THEOREM THEOREM %%%%%%%%%%%%%%%%%%
\begin{theorem}\label{t:VMOr-H}
The dual space of $\VMO_r(\R^n_+)$ is $H^1_z(\R^n_+)$.  
\end{theorem}
%%%%%%%%%%%%%%%%%%%%%%%%%% THEOREM THEOREM THEOREM %%%%%%%%%%%%%%%%%%

%%%%%%%%%%%%%%%%%%%%%%%%%%%%% Proof Proof Proof %%%%%%%%%%%%%%%%%%%%%%%%
\begin{proof}

For any $f \in H^1_z(\R^n_+)$, we define a linear operator by
\[
\mathscr{L}(g) := \int_{\R^n_+} f(x) g(x) dx, \ \ \forall g \in \VMO_r(\R^n_+), 
\]
where the integral in the right side of the equality denotes the dual form between $H^1_z(\R^n_+)$ and $\BMO_r(\R^n_+)$ in general. Let $G \in \VMO(\Rn)$ be an  arbitrary function such that $G|_{\R^n_+}=g$. It follows from Theorem 4.1 in \cite{CW} that the dual space of $\VMO(\Rn)$ is the space $H^1(\Rn)$. Thus, we have
\begin{align*}
|\mathscr{L}(g)| = \bigg|\int_{\Rn} f_z(x) G(x) dx \bigg|
\leq \|G\|_{\VMO(\Rn)} \|f_z\|_{H^1(\Rn)} =\|G\|_{\VMO(\Rn)} \|f\|_{H^1_z(\R^n_+)},
\end{align*}
and hence, 
\begin{align*}
|\mathscr{L}(g)| 
\leq \inf_{G \in \VMO(\Rn):G|_{\R^n_+}=g} \|G\|_{\VMO(\Rn)} \|f\|_{H^1_z(\R^n_+)}
=\|g\|_{\VMO_r(\R^n_+)} \|f\|_{H^1_z(\R^n_+)},
\end{align*}
which gives that
$||\mathscr{L}|| \leq \|f\|_{H^1_z(\R^n_+)}$.

For the converse, we follow the strategy in \cite[p. 639]{CW}. If $\langle f, g \rangle = 0$ for all $ f \in H^1_z(\R^n_+)$, then $g$ is the zero element of $\BMO_r(\R^n_+)$ and hence $g$ must be the zero of $\VMO_r(\R^n_+)$. Thus, $H^1_z(\R^n_+)$ is a total set of functionals on $\VMO_r(\R^n_+)$. We here need a theorem in the functional analysis from \cite[p. 439]{DS}. Let $X$ be a locally convex linear topological space and $Y$ be a linear subspace of $X^*$. Then $Y$ is $X$-dense in $X^*$ if and only if $Y$ is a total set of functionals on $X$. Thus, $H^1_z(\R^n_+)$ is weak* dense in $(\VMO_r(\R^n_+))^*$. It yields that for each $\mathscr{L} \in (\VMO_r(\R^n_+))^*$, there exists a sequence $\{f_k\} \subset H^1_z(\R^n_+)$ such that 
\begin{align*}
\mathscr{L}(g)=\lim_{k \to \infty} \int_{\R^n_+} f_k g ~ dx, \quad \forall g \in \VMO_r(\R^n_+). 
\end{align*}

Let us recall the Banach-Steinhaus theorem. Let $X$ be a Banach space and $Y$ a normed vector space. Suppose that $\Gamma$ is a collection of continuous linear operators from $X$ to $Y$. If for any $x \in X$ one has $\sup\limits_{T \in \Gamma} \|T(x)\|_{Y} < \infty$, then there holds 
\begin{align*} 
\sup_{T \in \Gamma} \|T\|_{X \to Y} =\sup_{T \in \Gamma,\|x\|_X=1} \|T(x)\|_{Y} < \infty. 
\end{align*}
From this, we obtain 
\begin{equation*}
\sup_{k \geq 1} \|f_k\|_{(\VMO_r(\R_+^n))^*} \le c_0\ \text{ for some } c_0>0,
\end{equation*}
and hence
\begin{equation}\label{e:sup-k}
\sup_{k \geq 1} \sup_{\substack{g \in C_c^\infty(\R_+^n) \\ 
\|g\|_{\BMO_r(\R_+^n)}\le 1}} \bigg|\int_{\R^n_+}f_k(x) g(x) dx\bigg|\le c_0.
\end{equation}
We shall show that for any $f\in H_z^1(\R_+^n)$, 
\begin{equation}\label{e:f-H1z}
\|f\|_{H^1_z(\R^n_+)}
\le \sup_{\substack{g \in C_c^\infty(\R_+^n) \\ \|g\|_{\BMO_r(\R_+^n)}\le 1}}
\bigg|\int_{\Rn}f(x)g(x) dx\bigg|. 
\end{equation}
Let $f\in H_z^1(\R_+^n)$. In view of Proposition 1.7 in \cite{CKS}, there exists an atomic decomposition as follows
\begin{equation*}
f=\sum_{j=1}^{\infty}\lambda_j a_j,
\end{equation*}
with $\|f\|_{H_z^1(\R_+^n)}\cong \sum_{j=1}^{\infty}|\lambda_j|$, where $\{a_j\}$ are standard $1$-atoms.  Let $\epsilon>0$, and fix $N \in \N_+$ satisfying $\sum_{j>N} |\lambda_j| < \epsilon $. Set 
\begin{equation*}
f^N := \sum_{j=1}^{N} \lambda_j a_j.
\end{equation*}
Then it holds 
\begin{equation*}
\|f\|_{H_z^1(\R_+^n)} \le \|f^N\|_{H_z^1(\R_+^n)}+\epsilon.
\end{equation*}
Since $(H_z^1(\R_+^n))^*=\BMO_r(\R_+^n)=\BMO_e(\R_+^n)$, by the Hahn-Banach theorem there exists $g\in \BMO_e(\R_+^n)$ such that 
\begin{equation*}
\|f^N\|_{H_z^1(\R_+^n)} = \langle g, f^N\rangle=\int_{\R_+^n} g(x)f^N(x) dx. 
\end{equation*}
(See, for example, Corollary IV-6-2, \cite[p.108]{Y}.) We know that there exist $c_1>0$ and $g_j\in C_c(\R_+^n)$ such that 
\begin{equation}\label{e:ggg}
\begin{split}
&g_j(\widetilde x)=g_j(x), \\
&\|g_j\|_{\BMO_e(\R_+^n)}\le c_1\|g\|_{\BMO_e(\R_+^n)},\\
&|g_j(x)|\le \sup_{0<r<1} \fint_{Q(x,r)}|g_e(y)|dy, \ x\in \R_+^n,\\
&\lim_{j\to\infty}g_j(x)=g_e(x), \ \text{ a.e }x\in \R_+^n.
\end{split}
\end{equation}
In fact, using the notations in Lemma \ref{lem:VMO-1}, we set 
\begin{equation*}
\varphi(x):=\varphi_0(|x|)\ 
\text{ and }\Phi_j(x):=\psi_{j^2}(2^{j^2+1}|x|/j),
\end{equation*}
Then we see that 
\begin{equation*}
\|\Phi_j\|_{\BMO(\Rn)}\le 4/{j^2},\ \text{ and } \Phi_j(x)=1 \ \text{ for }|x|\le j.
\end{equation*}
Define 
\begin{equation*}
g_j(x)=j^n\varphi(j\,\cdot\,)*[g_e]_j(x)\Phi_j(x),
\end{equation*}
where 
\begin{equation*}
[h(x)]_j=
\begin{cases}
h(x), & \text{ if } |h(x)| \le j, \\ 
j h(x)/|h(x)|, &\text{ if } |h(x)|>j. 
\end{cases}
\end{equation*}
One can easily check that these $g_j$ satisfy the desired properties.  Not only that, we can give an approximation as \eqref{e:ggg} for $\BMO$ functions in the general framework. This will be detailedly discussed in Section \ref{s:homo}. 

Applying this, we have
\begin{equation*}
\int_{\R_+^n} g(x)f^N(x) dx=\lim_{j\to\infty}\int_{\R_+^n}g_j(x)f^N(x) dx, 
\end{equation*}
which indicates that 
\begin{align*}
\|f^N\|_{H_z^1(\R_+^n)}
\le \sup_{\substack{g\in C_c(\R_+^n) \\ \|g\|_{\BMO_e(\R_+^n)}\le1}}
\biggl|\int_{\R_+^n} g(x)f^N(x) dx\biggr|
=\sup_{\substack{g\in C_c^\infty(\R_+^n) \\ \|g\|_{\BMO_r(\R_+^n)}\le1}}
\biggl|\int_{\R_+^n} g(x)f^N(x) dx\biggr|.
\end{align*}
Now, for $g\in C_c^\infty(\R_+^n)$, we get 
\begin{equation*}
\biggl|\int_{\R_+^n} g(x)f^N(x) dx\biggr|
\le \biggl|\int_{\R_+^n} g(x)f(x) dx\biggr|+\epsilon ,
\end{equation*}
and hence
\begin{equation*}
\sup_{\substack{g\in C_c^\infty(\R_+^n) \\ |g\|_{\BMO_r(\R_+^n)}\le1}}
\biggl|\int_{\R_+^n} g(x)f^N(x) dx\biggr|
\le \sup_{\substack{g\in C_c^\infty(\R_+^n)\\ \|g\|_{\BMO_r(\R_+^n)}\le1}} 
\biggl|\int_{\R_+^n} g(x)f(x) dx\biggr|+\epsilon.
\end{equation*}
Accordingly, there holds 
\begin{equation*}
\|f\|_{H_z^1(\R_+^n)}
\le \sup_{\substack{g\in C_c^\infty(\R_+^n)\\ \|g\|_{\BMO_r(\R_+^n)}\le1}}
\biggl|\int_{\R_+^n} g(x)f(x) dx\biggr|+2\epsilon.
\end{equation*}
Since $\epsilon $ is arbitrary, we obtain
\begin{equation*}
\|f\|_{H_z^1(\R_+^n)}
\le \sup_{\substack{g\in C_c^\infty(\R_+^n) \\ \|g\|_{\BMO_r(\R_+^n)}\le1}}
\biggl|\int_{\R_+^n} g(x)f(x) dx\biggr|.
\end{equation*}
This shows \eqref{e:f-H1z}. 

The inequalities \eqref{e:sup-k} and \eqref{e:f-H1z} conclude that
\begin{equation*}
\sup_{k \geq 1} \|f_k\|_{H^1_z(\R^n_+)} \le c_0. 
\end{equation*}
Since $\|f_{k,z}\|_{H^1(\Rn)}=\|f_k\|_{H^1_z(\R^n_+)}$, the sequence $\{f_{k,z}\}$ is bounded in $H^1(\Rn)$. By Lemma 4.2 in \cite{CW}, there exists a subsequence $\{f_{k_j,z} \}_{j \in \N_+}$ and $f \in H^1(\Rn)$ such that 
\begin{align}\label{e:fkj}
\lim_{j \to \infty} \int_{\Rn} f_{k_j,z} \phi ~ dx
=\int_{\Rn} f \phi ~ dx, \ \  
\forall \phi \in C_c^{\infty}(\Rn). 
\end{align} 
Taking $\phi \in C_c^{\infty}(\Rn)$ with $\supp(\phi) \subset \R^n_-$ in \eqref{e:fkj}, we obtain 
\begin{align*}
\int_{\R^n_-} f \phi ~ dx=0. 
\end{align*} 
By a similar argument as \eqref{e:Lf}, we conclude that 
\begin{equation*}
f(x)=0\ \text{ a.e. } x \in \R^n_-.  
\end{equation*}
Write $f_+:=f|_{\R^n_+}$. It means that $f_{+,z}=f$ a.e. $\Rn$ and $f_+ \in H^1_z(\R^n_+)$. In view of \eqref{e:fkj}, we deduce that 
\begin{align*}
\mathscr{L}(g)=\lim_{j \to \infty} \int_{\R^n_+} f_{k_j} g dx
=\int_{\R^n_+} f g dx, \ \  
\forall g \in C_c^{\infty}(\R^n_+). 
\end{align*} 
Thus, the linear functional $\mathscr{L} \in (\VMO_r(\R^n_+))^*$ is represented by $f_+ \in H^1_z(\R^n_+)$. 
\end{proof}
%%%%%%%%%%%%%%%%%%%%%%%%%%% End End Proof  %%%%%%%%%%%%%%%%%%%%%%%%%%%

%%%%%%%%%%%%%%%%%%%%%%% SECTION  SECTION SECTION %%%%%%%%%%%%%%%%%%%%%%
%%%%%%%%%%%%%%%%%%%%%%% SECTION  SECTION SECTION %%%%%%%%%%%%%%%%%%%%%%
\section{Compact commutators}\label{s:compact}

In this section, we will characterize $\VMO_{\Delta_N}(\Rn)$ via the compactness of commutators of Riesz transforms and the fractional integral operators associated with the Neumann Laplacian. 
%%%%%%%%%%%%%%%%%%%%%% SUBSECTION SUBSECTION SUBSECTION %%%%%%%%%%%%%%%%%
\subsection{Compactness of $[b, R_{N}]$}
The Riesz transforms associated to the Neumann Laplacian are given by
\[
R_N=(R_{N,1},\ldots,R_{N,n}):=\nabla \Delta_N^{-1/2}.
\]
The kernel of $R_{N,j}$ was formulated in \cite{LW} as
\begin{align*}
R_{N,j}(x,y) = \big(R_j(x,y) + R_j(x,\widetilde{y})\big) H(x_n y_n), j=1,\ldots,n,
\end{align*}
where $R_{j}(x,y)$ is the kernel of Riesz transform $R_j$:
\begin{align*}
R_{j}(x,y) = \frac{x_j - y_j}{|x-y|^{n+1}}, j=1,\ldots,n,
\end{align*}

%%%%%%%%%%%%%%%%%%%%%%%%%% Theorem Theorem Theorem %%%%%%%%%%%%%%%%%%%%%%
\begin{theorem}\label{t:bR-Lp}
Let $1<p<\infty$ and $j=1,\ldots,n$. Then $b \in \BMO_{\Delta_N}(\Rn)$ if and only if
$[b, R_{N,j}]$ is bounded on $L^p(\Rn)$. Moreover, we have 
\begin{align*}
\|[b, R_{N,j}]\|_{L^p(\Rn) \to L^p(\Rn)} \simeq \|b\|_{\BMO_{\Delta_N}(\Rn)}. 
\end{align*}
\end{theorem}
%%%%%%%%%%%%%%%%%%%%%%%%%% Theorem Theorem Theorem %%%%%%%%%%%%%%%%%%%%%%

%%%%%%%%%%%%%%%%%%%%%%%%%%%%% Proof Proof Proof %%%%%%%%%%%%%%%%%%%%%%%%
\begin{proof}
When $p=2$, the result was proved in \cite[Theorem 1.4]{LW}. But the proof was complicated because the authors used a weak  factorization of the space $H^1_{\Delta_N}(\Rn)$. Now we present a direct and easy proof for the lower bound and the upper bound can be obtained as the case $p=2$. 

Denote by $L^p_e(\Rn)$ the set of all functions $f \in L^p(\Rn)$ with $f(\widetilde{x})=f(x)$, $x \in \Rn$. Moreover, we endow the space $L^p_e(\Rn)$ with $L^p(\Rn)$ norm. Observing that 
\begin{align*}
R_j(\widetilde{x},y) &=R_j(x,y), \ \ 1 \leq j \leq n-1, \\ R_j(\widetilde{x},\widetilde{y}) &=-R_j(x,y), \ \ j=n, 
\end{align*}
one by changing variables has for any $f \in L^p_e(\Rn)$
\begin{align*}
[b_{+,e},R_j](f)(\widetilde{x})&=[b_{+,e},R_j](f)(x), \ \ 1 \leq j \leq n-1, \\ 
[b_{+,e},R_j](f)(\widetilde{x})&=-[b_{+,e},R_j](f)(x), \ \ j=n. 
\end{align*}
This indicates that 
\begin{align*}
\|[b_{+,e},R_j](f)\|_{L^p(\Rn)}^p
&=\int_{\R^n_+} |[b_{+,e},R_j](f)(x)|^p dx
+\int_{\R^n_-} |[b_{+,e},R_j](f)(x)|^p dx 
\\%%%%%%%%%
&=\int_{\R^n_+} |[b_{+,e},R_j](f)(x)|^p dx
+\int_{\R^n_+} |[b_{+,e},R_j](f)(\widetilde{x})|^p dx 
\\%%%%%%%%%
&=2\int_{\R^n_+} |[b_{+,e},R_j](f)(x)|^p dx. 
\end{align*}
Likewise, there holds 
\begin{align*}
\|[b_{-,e},R_j](f)\|_{L^p(\Rn)}^p
=2\int_{\R^n_-} |[b_{-,e},R_j](f)(x)|^p dx. 
\end{align*}
Additionally, we invoke \eqref{e:bRf} below to get 
\begin{align*}
\|[b,R_{N,j}](f)\|_{L^p(\Rn)}^p
&=\int_{\R^n_+} |[b_{+,e},R_j](f)(x)|^p dx + \int_{\R^n_-} |[b_{-,e},R_j](f)(x)|^p dx. 
\end{align*}
The above estimates conclude that 
\begin{equation*}
\|[b_{+,e},R_j]\|_{L_e^p(\Rn)\to L^p(\Rn)} + \|[b_{-,e},R_j]\|_{L_e^p(\Rn)\to L^p(\Rn)}
\lesssim \|[b,R_{N,j}]\|_{L^p(\Rn)\to L^p(\Rn)}. 
\end{equation*}
Consequently, the lower bound follows from this,  Theorem \ref{t:evenLp} below and Proposition \ref{p:BMO}. 
\end{proof}
%%%%%%%%%%%%%%%%%%%%%%%% End End End Proof %%%%%%%%%%%%%%%%%%%%%%%%%%%%

%%%%%%%%%%%%%%%%%%%%%%%%%% Theorem even Lp %%%%%%%%%%%%%%%%%%%%%%%%%%%
\begin{theorem}\label{t:evenLp}
Let $1<p<\infty$ and $b\in \bigcup_{1<q<\infty}L_{\rm{loc}}^q(\Rn)$ with $b(x)=b(\tilde x)$, $x \in \Rn$. Then for the Riesz transform $R_i$ $(i=1,\dots,n)$ there exists a constant $A=A(n,p,R_i)$ such that
\begin{equation*}
\|b\|_{\BMO(\Rn)} \le A \|[b,R_i]\|_{L_e^p(\Rn)\to L^p(\Rn)}. 
\end{equation*}
\end{theorem}
%%%%%%%%%%%%%%%%%%%%%%%%% end Theorem even Lp %%%%%%%%%%%%%%%%%%%%%%%%%%

%%%%%%%%%%%%%%%%%%%%%%%%%%% Proof Proof Proof %%%%%%%%%%%%%%%%%%%%%%%%%%
\begin{proof}
We may assume $\|[b,R_i]\|_{L_e^p(\Rn)\to L^p(\Rn)}=1$. We shall show that there exists $A>0$ such that for any cube $Q$
\begin{equation*}
\MO(b,Q):=\frac{1}{|Q|} \int_Q |b(x)-b_Q|dx \le A.
\end{equation*}
Considering  $b(x)=b(\tilde x)$, $x\in\Rn$, we may assume the $n$-th component of the center of $Q$ is nonnegative. Since $\|[b(r\cdot),R_i]\|_{L_e^p(\Rn)\to L^p(\Rn)} =\|[b,R_i]\|_{L_e^p(\Rn)\to L^p(\Rn)}$ for any $r>0$, we may assume $\ell(Q)=1/{\sqrt n}$. Since $[b-b_Q,\, R_i]=[b,\, R_i]$, we may also assume $b_Q=0$. 

Now, take a function $g\in L_c^\infty(\Rn)$ such that 
\begin{align}
\label{e:bdd-g-1} & |g(y)| \le 2, \ \ y \in Q,
\\%%%%%%%%%%
\label{e:bdd-g-2} & \supp(g) \subseteq Q, 
\\%%%%%%%%%%
\label{e:bdd-g-3} & \int_{\Rn} g(x) dx =0, 
\\%%%%%%%%%%
\label{e:bdd-g-4} & g(y)b(y) \geq 0, \ \ y \in Q,
\\%%%%%%%%%%
\label{e:bdd-g-5} & \int_{\Rn} b(x)g(x) dx =\MO(b,Q)=:M.
\end{align} 
Set 
\begin{align*}
f=g+\widetilde{g} \ \ \text{ and } \ \ \widetilde{g}(x)=g(\widetilde{x}). 
\end{align*} 
Then $\widetilde{g}$ also satisfies \eqref{e:bdd-g-1}-\eqref{e:bdd-g-5} and $f \in L^p_e(\Rn)\cap L^\infty_e(\Rn)$.  

We here need an elementary observation. Let $1\le i\le n$, and $Q$ be a cube with center $x_0$. Suppose $\displaystyle \frac{y_i-(x_0)_i}{|y-x_0|}>2A_1$ with $0<A_1<1/2$, and $|y-x_0|>\sqrt n\,(2/A_1+1)\ell(Q)$. Then, it holds
\begin{equation}\label{e:yz}
\frac{y_i-z_i}{|y-z|}\ge A_1, \ \text{ for } z\in Q,
\end{equation}
and $y_i-z_i\ge 0$ for $z\in \widetilde{Q}$ with $y_n\ge 0$. Denote 
\begin{align*}
G_{i}:=\Big\{x\in \Rn: x_n\ge0,\, |x-x_0|>A_2:=\frac{2}{A_1}+1, \text{ and } 
\frac{x_i-(x_0)_i}{|x-x_0|}>2A_1\Big\}
\end{align*}
If $x\in G_{i}$, we have $(x_i-z_i)/|x-z|\ge A_1$ for $z\in Q$ and $x_i-z_i\ge 0$ for $z\in \widetilde{Q}$. Applying \eqref{e:bdd-g-2}, \eqref{e:bdd-g-4} and \eqref{e:bdd-g-5},  we have 
\begin{align*}
R_i(bf)(x)&=\int_\Rn \frac{x_i-z_i}{|x-z|^{n+1}}b(z)f(z)dz
\\%%%%%%%%%%
&=\int_Q \frac{x_i-z_i}{|x-z|^{n+1}}b(z)g(z)dz
+\int_{\widetilde{Q}} \frac{x_i-z_i}{|x-z|^{n+1}}b(z)\widetilde{g}(z)dz
\\%%%%%%%%%%
&\ge \frac{A_3}{|x-x_0|^n}\int_Q b(z)g(z)dz
=\frac{A_3}{|x-x_0|^n}M.
\end{align*}
Using \eqref{e:bdd-g-1}, \eqref{e:bdd-g-2} and \eqref{e:bdd-g-3}, we get for $x\in G_{i}$
\begin{equation*}
|b(x)R_i(f)(x)|\le A_4 \frac{|b(x)|}{|x-x_0|^{n+1}}.
\end{equation*}
Thus we have for $x\in G_{i}$
\begin{equation*}
|[b,R_i](f)(x)|\ge \frac{A_3}{|x-x_0|^n}M- \frac{A_4|b(x)|}{|x-x_0|^{n+1}}.
\end{equation*}
Write 
\begin{equation*}
F_i=\Bigl\{x\in G_i: |b(x)|>\frac{MA_3}{2A_4}|x-x_0|,\,|x-x_0|<M^{p'/n}
\Bigr\}.
\end{equation*}
Now we may assume $A_2^n<M^{p'}/2$, otherwise we have $M\le (2A_2)^{1/p}$ and nothing to prove. By the definition of $F_i$, we get $|F_i|<v_n M^{p'}$, where $v_n$ is the volume of the unit ball in $\Rn$. And hence we have $A_2^n+|F_i|/(2v_n) <M^{p'}$. Since 
$$
(G_i \setminus F_i) \cap \{x \in \Rn: |x|<M^{p'/n}\} 
\supset G_i \cap \{x \in \Rn: A_2<|x-x_0|<M^{p'/n}\}
$$
and the latter contains 
$G_i \cap \{x\in\Rn: (A_2^n+|F_i|/(2v_n))^{1/n}<|x-x_0|<M^{p'/n}\}$, 
we arrive at 
\begin{align*}
2\cdot2^p\ge\|f\|^p_{L^p(\Rn)}
&\ge \int_{\Rn}|[b,R_i]f(x)|^pdx
\\%%%%%%%%%%
&\ge \int_{(G_i \setminus F_i) \cap \{|x-x_0|<M^{p'/n}\}}
\bigg(\frac{MA_3}{2|x-x_0|^n}\bigg)^p dx
\\%%%%%%%%%%
&\ge \int_{G_i \cap \{A_5(|F_i|+A_2^n)^{1/n}<|x-x_0|<M^{p'/n}\}}
\bigg(\frac{MA_3}{2|x-x_0|^n}\bigg)^pdx
\\%%%%%%%%%%
&=\bigg(\frac{A_3M}{2}\bigg)^p|\Lambda| 
\int_{A_5(|F_i|+A_2^n)^{1/n}}^{M^{p'/n}} t^{-np+n-1}dt
\\%%%%%%%%%%
&=\bigg(\frac{A_3M}{2}\bigg)^p\frac{|\Lambda|}{n(1-p)}
\big(M^{p'(1-p)} -A_5^{n(1-p)}(|F_i|+A_2^{n})^{1-p}\big), 
\end{align*}
where $\Lambda:=\{x\in \mathbb{S}^{n-1}; {x_i-(x_0)_i}>2A_1\}$. 

From this we obtain 
\begin{equation*}
|F_i|\ge 2v_n(A_6M^{p'}-A_2^n)\ge v_nA_6M^{p'},
\end{equation*}
if $M>(2A_2^n/A_6)^{1/p'}$ ( otherwise nothing to prove). 

Now take $h(x)=\operatorname{sgn}(b(x))\,\mathbf{1}_{F_i}$. Denote by $S_i$ the adjoint operator of $R_i$ from $L_e^p(\Rn)$ to $L^p(\Rn)$. Then
\begin{align*}
S_i(x,y)&=\frac{1}{2}\bigl(R_i(x-y)+R_i(\widetilde x -y)\bigr),\ 1\le i\le n-1,
\\%%%%%%%%%%
S_n(x,y)&=\frac{1}{2}\bigl(R_n(x-y)-R_n(\widetilde x -y)\bigr).
\end{align*}
For $x\in Q$ we have
\begin{equation*}
|[b,S_i]h(x)|\ge \biggl|\int_{F_i}S_i(y,x)|b(y)|dy\biggr|
-|b(x)|\biggl|\int_{\Rn}S_i(y,x)h(y)dy\biggr|.
\end{equation*}
Since $y\in F_i$, we get 
\begin{equation*}
|b(y)|>\frac{MA_3}{2A_4}|y-x_0|. 
\end{equation*}
When $1\le i\le n-1$
\begin{equation*}
\int_{F_i}S_i(y,x)|b(y)|dy
=\frac{1}{2}\int_{F_i}R_i(y-x)|b(y)|dy 
+\frac{1}{2}\int_{\Rn}R_i(\widetilde y -x)|b(y)|dy.
\end{equation*}
For $x\in Q$ and $y\in F_i$, by \eqref{e:yz} we get 
$\displaystyle{\frac{y_i-x_i}{|y-x|}\ge A_1}$, and from $\widetilde y\in 
\widetilde {F_i}$ and $x\in Q$ we have $y_i-x_i\ge0$. 
Hence we obtain
\begin{align*}
\biggl|\int_{F_i}S_i(y,x)|b(y)|dy\biggr|
&\ge \frac{1}{2}\int_{F_i}\frac{A_1}{|y-x|^n}\frac{MA_3}{2A_4}|y-x_0|dy
\\%%%%%%%%%%
&\ge A_7M \int_{F_i}|y-x_0|^{1-n}dy \ge A_8M^{1+p'/n}.  
\end{align*}
On the other hand, from the definition of $F_i$ we see that $|y-x|\le |\widetilde y-x|$ for $y\in F_i$ and $x\in Q$. So, there exists 
$A_9$ such that 
\begin{align*}
\biggl|\int_{\Rn}S_i(y,x)h(y)dy\biggr|
&\le \frac12\int_{F_i}\frac{dy}{|y-x|^n}
+\frac12\int_{F_i}\frac{dy}{|\tilde y-x|^n}
\le A'_8\int_{F_i}\frac{dy}{|y-x_0|^n}
\\%%%%%%%%%
&\le A'_8\int_{A_2<|y-x_0|<M^{p'/n}}\frac{dy}{|y-x_0|^n}
\le A_9 \log M. 
\end{align*}
Thus, we obtain 
\begin{equation}\label{eq:BMOComm-10}
|[b, S_i]h(x)|>A_8M^{1+p'/n}-A_9 \log M. 
\end{equation}
Since $[b, S_i]$ is the adjoint operator of $[b,R_i]$, we have
\begin{equation*}
\|[b, S_i]\|_{L^{p'}(\Rn)\to L_e^{p'}(\Rn)}\le 1.
\end{equation*}
It follows from the definition of $F_i$ and \eqref{eq:BMOComm-10} that 
\begin{align*}
A_{10}M&\ge \|h\|_{L^{p'}(\Rn)}
\ge \|[b, S_i]h\|_{L_e^{p'}(\Rn)}
\ge \int_{Q}|[b, S_i]h(x)|dx
\\%%%%%%%%%
&\ge \int_{Q}\bigl(A_8M^{1+p'/n}-A_9\log M\bigr)dx
=|Q|\bigl(A_8M^{1+p'/n}-A_9\log M\bigr).
\end{align*}
Then we have $M\le A(n,p,R_i)$. 

In the case $i=n$
\begin{equation*}
\int_{F_i}S_n(y,x)|b(y)|dy =\frac{1}{2}\int_{F_i}R_n(y-x)|b(y)|dy 
-\frac{1}{2}\int_{\Rn}R_n(\widetilde y -x)|b(y)|dy.
\end{equation*}
For $x\in Q$ and $y\in F_n$, we by \eqref{e:yz} get $\displaystyle{\frac{y_n-x_n}{|y-x|}\ge A_1}$, and from $\widetilde y\in 
\widetilde{F}_n$ and $x\in Q$ we have $y_n-x_n\le0$. Hence we obtain
\begin{align*}
\biggl|\int_{F_n}S_n(y,x)|b(y)|dy\biggr|
&\ge \frac{1}{2} \int_{F_n} \frac{A_1}{|y-x|^n} 
\frac{MA_3}{2A_4}|y-x_0|dy
\\%%%%%%%%% 
&\ge A_7M\int_{F_n}|y-x_0|^{1-n}dy\ge A_8M^{1+p'/n}.  
\end{align*}
Proceeding the proof in the same way as in the case $1\le i\le n-1$ hereafter, we get the desired estimate. This completes the proof of Theorem \ref{t:evenLp}. 
\end{proof}
%%%%%%%%%%%%%%%%%%%%%%% end Proof of Theorem even Lp %%%%%%%%%%%%%%%%%%%%%%%%

%%%%%%%%%%%%%%%%%%%%%%%%%%% Theorem Theorem Theorem %%%%%%%%%%%%%%%%%%%%%
\begin{theorem}\label{t:compact-1}
Let $1<p<\infty$ and $j=1,\ldots,n$. Then $b \in \VMO_{\Delta_N}(\Rn)$ if and only if $[b, R_{N,j}]$ is a compact operator on $L^p(\Rn)$.
\end{theorem}
%%%%%%%%%%%%%%%%%%%%%%%%%%% Theorem Theorem Theorem %%%%%%%%%%%%%%%%%%%%%

%%%%%%%%%%%%%%%%%%%%%%%%%%%%% Proof Proof Proof %%%%%%%%%%%%%%%%%%%%%%%%
\begin{proof}
We first show that $[b, R_{N,j}]$ is a compact operator on $L^p(\Rn)$ if $b \in \VMO_{\Delta_N}(\Rn)$. Let $\{f_k\}_{k=1}^{\infty}$ be an arbitrary uniformly bounded sequence in $L^p(\Rn)$. In order to prove the compactness of $[b, R_{N,j}]$, it only needs to select a subsequence $\{f_{k_l}\}_{l=1}^{\infty} \subseteq \{f_k\}_{k=1}^{\infty}$ such that $\{[b, R_{N,j}](f_{k_l})\}_{l=1}^{\infty}$ converges in $L^p(\Rn)$.

By Theorem \ref{t:VMO}, we have $b_{+,e} \in \VMO(\Rn)$ and $b_{-,e} \in \VMO(\Rn)$ provided by $b \in \VMO_{\Delta_N}(\Rn)$. Then it follows from \cite[Theorem 2]{U} that both $[b_{+,e},R_j]$ and $[b_{-,e},R_j]$ are compact operators on $L^p(\Rn)$. Note that $\{f_k\}_{k=1}^{\infty}$ is uniformly bounded in $L^p(\Rn)$, and
\[
\|f_{k,+,e}\|_{L^p(\Rn)}+\|f_{k,-,e}\|_{L^p(\Rn)}
\lesssim \|f_k\|_{L^p(\Rn)}, \ \ \forall k \geq 1.
\]
This gives that $\{f_{k,+,e}\}_{k=1}^{\infty}$ and $\{f_{k,-,e}\}_{k=1}^{\infty}$ are uniformly bounded in $L^p(\Rn)$. In view of the compactness of $[b_{+,e},R_j]$, there exists a sequence $\{k_m\}_{m=1}^{\infty}$ such that
\begin{align}\label{e:km}
\{[b_{+,e},R_j](f_{k_m,+,e})\}_{m=1}^{\infty}
\ \text{ converges in } L^p(\Rn).
\end{align}
Observer that $\{f_{k_m,-,e}\}_{m=1}^{\infty}$ is uniformly bounded in $L^p(\Rn)$. Together with the compactness of $[b_{-,e},R_j]$, there exists a subsequence $\{k_{m_l}\}_{l=1}^{\infty} \subseteq \{k_m\}_{m=1}^{\infty}$ such that
\begin{align}\label{e:kml-}
\big\{[b_{-,e},R_j](f_{k_{m_l},-,e}) \big\}_{l=1}^{\infty}\quad\text{converges in } L^p(\Rn).
\end{align}
The equation \eqref{e:km} implies that
\begin{align}\label{e:kml+}
\left\{[b_{+,e},R_j](f_{k_{m_l},+,e}) \right\}_{l=1}^{\infty}
\ \text{ converges in } L^p(\Rn).
\end{align}

For any $x \in \Rn$ and function $f$ on $\Rn$, we have
\begin{align*}
R_{N,j}(f\mathbf{1}_{\R^n_+})(x)
&=\int_{\R^n_+} (R_j(x,y)+R_j(x,\widetilde{y})) H(x_n y_n) f(y) dy
\\ %%%%%%%%%%%%%%%
&=\mathbf{1}_{\R^n_+}(x) \int_{\R^n_+} R_j(x,y) f(y) dy
+ \mathbf{1}_{\R^n_+}(x) \int_{\R^n_-} R_j(x,y) f(\widetilde{y}) dy
\\ %%%%%%%%%%%%%%%
&=\mathbf{1}_{\R^n_+}(x) \int_{\Rn} R_j(x,y) f_{+,e}(y) dy
=\mathbf{1}_{\R^n_+}(x) R_j(f_{+,e})(x),
\end{align*}
which gives that
\begin{equation}\label{e:bRf}
\begin{split}
[b,R_{N,j}](f\mathbf{1}_{\R^n_+})
&= b R_{N,j}(f\mathbf{1}_{\R^n_+}) - R_{N,j}(bf\mathbf{1}_{\R^n_+}) \\
&= [b_{+,e} R_{j}(f_{+,e}) - R_{j}(b_{+,e}f_{+,e})]\mathbf{1}_{\R^n_+} 
=[b_{+,e},R_j](f_{+,e}) \mathbf{1}_{\R^n_+}.
\end{split}
\end{equation}
Similarly, it holds
\begin{equation*}
[b,R_{N,j}](f\mathbf{1}_{\R^n_-}) = [b_{-,e},R_j](f_{-,e})\mathbf{1}_{\R^n_-}. 
\end{equation*}
As a consequence, we deduce that
\begin{align*}
[b,R_{N,j}](f_{k_{m_l}}) & = [b,R_{N,j}](f_{k_{m_l}}\mathbf{1}_{\R^n_+})
+ [b,R_{N,j}](f_{k_{m_l}}\mathbf{1}_{\R^n_-})
\\ %%%%%%%%%%%%%%%
&=[b_{+,e},R_j](f_{k_{m_l},+,e})\mathbf{1}_{\R^n_+}
+ [b_{-,e},R_j](f_{k_{m_l},-,e})\mathbf{1}_{\R^n_-}.
\end{align*}
From \eqref{e:kml+} and \eqref{e:kml-}, it immediately implies that
\begin{equation*}
\left\{[b,R_{N,j}](f_{k_{m_l}})\right\}_{l=1}^{\infty} \ \text{ converges in } L^p(\Rn).
\end{equation*}
This shows that $[b, R_{N,j}]$ is compact on $L^p(\Rn)$.

Let us turn our attention to the sufficiency. Let $[b, R_{N,i}]$ be a compact operator on $L^p(\Rn)$, $(1\le i\le n)$. Then $[b, R_{N,i}]$ is bounded on $L^p(\Rn)$, which by Theorem \ref{t:bR-Lp}  is equivalent to $b_{+,e} \in \BMO(\Rn)$ and $b_{-,e} \in \BMO(\Rn)$. In view of Theorem \ref{t:VMO}, it is enough to prove that $b_{+,e} \in \VMO(\Rn)$ and $b_{-,e} \in \VMO(\Rn)$. Here we will only focus on the first one.

Recall that $L^p_e(\Rn)$ is a closed subspace of $L^p(\Rn)$.  Then $[b, R_{N,i}]$ is a compact operator from $L^p_e(\Rn)$ to 
$L^p(\Rn)$.  Applying the same argument as that in the preceding paragraph, we can show that $[b_{+,e}, R_{i}]$ is also a compact operators from $L^p_e(\Rn)$ to $L^p(\Rn)$.

We assume that $\|b_{+,e}\|_{\BMO(\Rn)}=1$. In order to show $b_{+,e} \in \VMO(\Rn)$, we use a contradiction argument
via Proposition \ref{p:VMO}. Suppose that $\gamma_1(b_{+,e})>0$. In the sequel, we simply write $\b$ for $b_{+,e}$.
Then $\b(\widetilde{x})=\b(x)$, $x \in \Rn$. By $\gamma_1(\b)>0$, there exists $\delta>0$ and a sequence of cubes $\{Q_j\}_{j=1}^{\infty}$ such that $\lim\limits_{j \to \infty} \ell(Q_j)=0$ and $\MO(\b,Q_j) > \delta$, $j \in \N_+$. The fact $\b(\widetilde{x})=\b(x)$ implies that
\begin{align*}
\b_{Q_j}=\b_{\widetilde{Q}_j}
\ \text{ and } \ \MO(\b,Q_j)=\MO(\b,\widetilde{Q}_j), \ \ j \in \N_+.
\end{align*}
Thus, we may assume that the $n$-th component of $x_j$ is nonnegative, $j \in \N_+$, where $x_j$ is the center of $Q_j$.

Take $0<A_1<1/2$ and denote
\begin{align*}
\Lambda &:=\left\{x' \in \S^{n-1}: \Omega(x)>2A_1 \right\},
\\ %%%%%%%%%%
G_{\ell} &:= \left\{x \in \Rn: |x|>\sqrt{n} \ell A_2 \
\text{ and } \ x' \in \Lambda \right\},
\end{align*}
where $\Omega(x)=\frac{x_i}{|x|}$ and  $A_2=1+\frac{2}{A_1}$. For $j\in\N$, set
\begin{equation*}
g_j(x)=|Q_j|^{-\frac1p} \Bigl(\sgn(\b(x)-\b_j)-(\sgn(\b-\b_j))_{Q_j}\Bigr)
\mathbf{1}_{Q_j}(x),
\end{equation*}
where $\b_j=\b_{Q_j}$. Then we can easily check that $g_j$ satisfies the following: 
\begin{align}
\label{e:g-1} & g_j(x)(\b(x)-\b_j) \geq 0,\ x\in \Rn, \\
\label{e:g-2} & \supp(g_j) \subseteq Q_j, \\
\label{e:g-3} & \int_{\Rn} g_j(x) dx =0, \\
\label{e:g-4} & |g_j(x)| \le 2|Q_j|^{-1/p}, \ \ x \in Q_j.\\
\label{e:g-5} & \int_{Q_j} (\b(x)-\b_j)g_j(x) dx =|Q_j|^{1-\frac1p} \MO(\b, Q_j).
\end{align}
Set
\begin{align*}
f_j=g_j+\widetilde{g}_j, \ \ \widetilde{g}_j(x)=g_j(\widetilde{x}).
\end{align*}
Then $\widetilde{g}_j$ also satisfies \eqref{e:g-1}-\eqref{e:g-4} and $f_j \in L^p_e(\Rn)$. It follows from \eqref{e:g-2}-\eqref{e:g-4} that
\begin{align*}
|R_i((\b-\b_j)f_j)(y)|
&\leq |R_i((\b-\b_j)g_j)(y)|
+ |R_i((\b-\b_j)\widetilde{g}_j)(y)|
\\ %%%%%%%%%
&\leq A_{20} \frac{|Q_j|^{1-\frac1p}}{|y-x_j|^n}, \ \ \
y \not\in A_{21} Q_j.
\end{align*}
Let $y-x_j \in G_{\ell(Q_j)}$ and $z \in Q_j$. Then we have $(y_i-z_i)/|y-z| \geq A_1$. Indeed, together with $|y-x_j| \geq \sqrt{n} A_2 \ell(Q_j)/2$, the fact $|z-x_j| \leq \sqrt{n} \ell(Q_j)/2$ implies that $|y-x_j| \geq 2A_2 |z-x_j|$ and $|y-z| \geq |y-x_j|-|z-x_j| \geq (2A_2-1)|z-x_j|$. Then we see that
\begin{align*}
1-\frac{1}{2A_2} \leq \frac{|y-z|}{|y-x_j|} \leq 1+\frac{1}{2A_2}
\end{align*}
provided by $|y-x_j|-|z-x_j| \leq |y-z| \leq |y-x_j|+|z-x_j|$. This immediately gives that
\begin{align*}
\frac{y_i-z_i}{|y-z|} &= \frac{y_i-(x_j)_i}{|y-x_j|} \frac{|y-x_j|}{|y-z|} - \frac{z_i-(x_j)_i}{|y-z|}
\\ %%%%%%%%%%
&\geq 2A_1 \cdot \frac{1}{1+\frac{1}{2A_2}} - \frac{1}{2A_2-1}
\\ %%%%%%%%%%
&= 2A_1 - 2A_1 \cdot \frac{\frac{1}{2A_2}}{1+\frac{1}{2A_2}} - \frac{1}{2A_2-1}
\\ %%%%%%%%%%
&= 2A_1 - \frac{2A_1^2}{3A_1+4} - \frac{A_1}{A_1+4}
\\ %%%%%%%%%%
&\geq 2A_1 - \frac{2}{3}A_1 - \frac14 A_1 = \frac{13}{12} A_1.
\end{align*}
As a consequence, we deduce by \eqref{e:g-5} that
\begin{align*}
R_i((\b-\b_j)g_j)(y) &= \int_{Q_j} \frac{y_i-z_i}{|y-z|^{n+1}} (\b(z)-\b_j) g_j(z) dz
\\ %%%%%%%%%%
&\geq \frac{cA_1}{|y-x_j|^n} \int_{Q_j} (\b(z)-\b_j) g_j(z) dz
\\ %%%%%%%%%%
&=cA_1 |Q_j| \MO(\b,Q_j) \frac{|Q_j|^{-\frac1p}}{|y-x_j|^n}
>c \delta A_1 \frac{|Q_j|^{1-\frac1p}}{|y-x_j|^n}.
\end{align*}
Additionally, we have
\begin{align*}
R_i((\b-\b_j)\widetilde{g}_j)(y)
&= \int_{\widetilde{Q}_j} \frac{y_i-z_i}{|y-z|^{n+1}} (\b(z)-\b_j)\widetilde{g}_j(z) dz
\\%%%%%%%%%%
&=\int_{{Q}_j} \frac{y_i-(\widetilde z)_i}{|y-\widetilde z|^{n+1}} (\b(z)-\b_j) {g}_j(z) dz.
\end{align*}
When $1\le i\le n-1$, we have $(\widetilde z)_i=z_i$, and so $y_i-(\widetilde z)_i=y_i-z_i>0$ for $z\in Q_j$ because of $(y_i-z_i)/|y-z|\ge A_1$. When $i=n$, we get $(\widetilde z)_i=-z_n$, and hence $y_i-(\widetilde z)_i=y_n+z_n$. We easily observe that $y_n+z_n>0$ if $y_n>0$ and $z\in Q_j$.

Thus, it holds for $y \in x_j+G_{\ell(Q_j)}$ with $y_n>0$ that
\begin{align*}
R_i((\b-\b_j)f_j)(y) >c \delta A_1 \frac{|Q_j|^{1-\frac1p}}{|y-x_j|^n}.
\end{align*}
Let
\[
R^+(x,\alpha,\beta)=\{y \in \R^2: y_2 \geq 0, \alpha<|x-y|<\beta\}.
\]
Replacing $R(x,\alpha \ell(Q_j),\beta \ell(Q_j))$ by $R^+(x,\alpha \ell(Q_j),\beta \ell(Q_j))$ in \cite[p. 169]{U}, we obtain
\begin{align*}
&\|[\b,R_i]f_{j(h)} - [\b,R_i]f_{j(h+m)}\|_{L^p(\Rn)}^p
\\%%%%%%%%%%
&\geq \int_{R^+(x_{j(h)},A_{29}\ell(Q_{j(h)}), A_{30} \ell(Q_{j(h)}))
\setminus
R^+(x_{j(h+m)},0, A_{30} \ell(Q_{j(h+m)}))} \cdots dx
\\%%%%%%%%%%
&\geq \bigg(\Big(\frac12\Big)^{1/p} - \Big(\frac14\Big)^{1/p} \bigg)^p A_{31}.
\end{align*}
This indicates that $[\b, R_i]$ is not a compact operator from $L^p_e(\Rn)$ to $L^p(\Rn)$.

Similar consideration holds for the cases $\gamma_2(\b)>0$ and $\gamma_3(\b)>0$.
\end{proof}
%%%%%%%%%%%%%%%%%%%%% End End End Proof of Theorem 5.2 %%%%%%%%%%%%%%%%%%%%%%%%

%%%%%%%%%%%%%%%%%%%%%%%% SUBSECTION SUBSECTION SUBSECTION %%%%%%%%%%%%%%%
\subsection{Compactness of $[b, \Delta_{N}^{-\alpha/2}]$} 
For $0<\alpha<n$, the fractional operator $\Delta_{N}^{-\alpha/2}$ of $\Delta_{N}$ is defined by
\begin{align*}
\Delta_{N}^{-\alpha/2} f(x)
:= \frac{1}{\Gamma(\alpha/2)} \int_{0}^{\infty} e^{- t \Delta_N}f(x) \frac{dt}{t^{1-\alpha/2}}.
\end{align*}
If we by $K_{N,\alpha}(x,y)$ denote the kernel of $\Delta_{N}^{-\alpha/2}$, then it follows from \cite[Proposition 2.4]{LW} that
\begin{align*}
K_{N,\alpha}(x,y) := \bigg(\frac{C_{n,\alpha}}{|x-y|^{n - \alpha}}
+ \frac{C_{n,\alpha}}{(|x'-y'|^2 + |x_n + y_n|^2)^{\frac{n}{2} - \frac{\alpha}{2}}} \bigg) H(x_n y_n),
\end{align*}
where $C_{n,\alpha} = \frac{1}{2^{\alpha} \pi^{n/2}} \frac{\Gamma((n-\alpha)/2)}{\Gamma(\alpha/2)}$.

%%%%%%%%%%%%%%%%%%%%%%%% Theorem Theorem Theorem %%%%%%%%%%%%%%%%%%%%%%%%
\begin{theorem}\label{t:compact-2}
Let $0 < \alpha < n$, $ 1 < p < q < \infty $ with $\frac{1}{q} = \frac{1}{p} - \frac{\alpha}{n}$. Then $b \in \VMO_{\Delta_N}(\Rn)$ if and only if $[b, \Delta_{N}^{-\alpha/2}]$ is a compact operator from $L^p(\Rn)$ to $L^q(\Rn)$.
\end{theorem}
%%%%%%%%%%%%%%%%%%%%%%%%% Theorem Theorem Theorem %%%%%%%%%%%%%%%%%%%%%%%

%%%%%%%%%%%%%%%%%%%%%%%%%%%% Proof Proof Proof %%%%%%%%%%%%%%%%%%%%%%%%%
\begin{proof}
Since the proof is similar to that for Theorem \ref{t:compact-1}, we only mention certain steps and omit the details. We will recur to the classical Riesz potential, which is defined by
\[
\Delta^{-\alpha/2}f(x) :=\frac{1}{\Gamma(\alpha/2)} \int_{0}^{\infty} e^{- t \Delta}f(x) \frac{dt}{t^{1-\alpha/2}}
=C_{n,\alpha}\int_{\Rn} \frac{f(y)}{|x-y|^{n-\alpha}}dy.
\]
In \cite{W}, Wang proved that $b \in \VMO(\Rn)$ if and only if $[b,\Delta^{-\alpha/2}]$ is a compact operator from $L^p(\Rn)$ to $L^q(\Rn)$. Moreover, Theorem \ref{t:VMO} plays an important role in the connection between $\VMO_{\Delta_N}(\Rn)$ and $\VMO(\Rn)$.

As we have seen, in order to prove the compactness, we have to first establish the boundedness of $[b,\Delta_{N}^{-\alpha/2}]$. In fact, one can characterize $\BMO_{\Delta_N}(\Rn)$ via commutator:
\[
\|[b,\Delta_{N}^{-\alpha/2}]\|_{L^p(\Rn) \to L^q(\Rn)} \simeq ||b||_{\BMO_{\Delta_N}(\Rn)}.
\]
The upper bound was contained in \cite[Theorem 1.5]{LW}. Thus, it suffices to show the lower bound. Applying the same technique to get $(7.4)$ in \cite{DHLWY}, we deduce that
\begin{align*}
\|[b_{+,e},\Delta^{-\alpha/2}]\|_{L^p(\Rn) \to L^q(\Rn)}
+ \|[b_{-,e},\Delta^{-\alpha/2}]\|_{L^p(\Rn) \to L^q(\Rn)}
\lesssim \|[b,\Delta_{N}^{-\alpha/2}]\|_{L^p(\Rn) \to L^q(\Rn)}.
\end{align*}
We invoke a result
\[
\|[a,\Delta^{-\alpha/2}]\|_{L^p(\Rn) \to L^q(\Rn)} \simeq ||a||_{BMO(\Rn)}.
\]
It is contained in \cite{HRS}, in which the authors obtained the equivalence between the weighted $\BMO(\Rn)$ and the two-weight inequality for the commutator of Riesz potential. Combining these two inequalities, it yields that
\begin{align*}
||b||_{\BMO_{\Delta_N}(\Rn)} \simeq \|b_{+,e}\|_{\BMO(\Rn)} + \|b_{-,e}\|_{\BMO(\Rn)}
\lesssim \|[b,\Delta_{N}^{-\alpha/2}]\|_{L^p(\Rn) \to L^q(\Rn)}.
\end{align*}
The proof is concluded. 
\end{proof}
%%%%%%%%%%%%%%%%%%%%%%%%%%% End End End Proof %%%%%%%%%%%%%%%%%%%%%%%%%

%%%%%%%%%%%%%%%%%%%%% SECTION  SECTION SECTION %%%%%%%%%%%%%%%%%%%%%%%%%
%%%%%%%%%%%%%%%%%%%%% SECTION  SECTION SECTION %%%%%%%%%%%%%%%%%%%%%%%%%
\section{$\BMO$ approximation}\label{s:homo}
Let $(X,d,\mu)$ be a space of homogeneous type in the sense of Coifman-Weiss. That is, $X$ is a topological space endowed with a Borel measure $\mu$ and a quasi-metric $d$, satisfying the following conditions: (a) $d(x,y)=d(y,x)$, (b) $d(x,y)>0$ if and only if $x\ne y$ and (c) there exists a constant $K$ such that $d(x,y)\le K[d(x,z)+d(z,y)]$ for all $x,y,z\in X$. (d) the balls $B(x,r)=\{y\in X;\,d(x,y)<r\}$ centered at $x$ and of radius $r>0$ form a basis of open neighborhoods of the point $x$ 
and, also, $\mu(B(x,r))>0$ whenever $r>0$. Furthermore, $\mu$ satisfies the doubling condition: there exists a positive 
constant $A$ such that $\mu(B(x,2r))\le A\,\mu(B(x,r))$.

The purpose of this section is to give an approximation for $\BMO(X)$ functions by the continuous functions with bounded supports as follows. We have seen an application of such approximation in Section \ref{s:dual}. We also believe that there will be more applications of it. 

%%%%%%%%%%%%%%%%%%%%% Proposition Proposition Proposition %%%%%%%%%%%%%%%%%%%%%%%
\begin{proposition}\label{p:appr-BMO}
For any $f\in {\BMO}(X)$ there exists a sequence of bounded, continuous and boundedly supported $\{f_j\}_{j=1}^{\infty}$ such that 
\begin{align*}
&\|f_j\|_{{\BMO}}\le a_1 \|f\|_{{\BMO}},
\\%%%%%%%%%
&|f_j(x)|\le a_2 \mathbb{M}f(x), \  \ x \in X, 
\\%%%%%%%%%
&\lim_{j\to\infty}f_j(x)=f(x), \ \text{ a.e. } x \in X,
\end{align*}
where $a_1$ and $a_2$ are independent on $f$, and $\mathbb{M}$ is the restricted centered Hardy-Littlewood maximal function of $f$: 
\begin{align*}
\mathbb{M}f(x) 
=\sup_{0<r<1} \frac{1}{\mu(B(x,r))} \int_{B(x,r)}|f(y)|\,d\mu(y). 
\end{align*} 
\end{proposition}
%%%%%%%%%%%%%%%%%%%%% Proposition Proposition Proposition %%%%%%%%%%%%%%%%%%%%%%%

%%%%%%%%%%%%%%%%%%%%%%%% Remark Remark Remark %%%%%%%%%%%%%%%%%%%%%%%%%
\begin{remark}
If $(X,d,\mu)$ is complete as a quasi-metric space, the closure of any ball is compact, because of its total boundedness, which can be seen by using Theorem (3.1) and the claim (3.4) in \cite{CW}. Hence, the functions $f_j$ above are compactly supported. 
\end{remark}
%%%%%%%%%%%%%%%%%%%%%%%% Remark Remark Remark %%%%%%%%%%%%%%%%%%%%%%%%%

We note the following: It is known in \cite{MacSeg} that there exist a quasi-distance $d'(x,y)$ on $X$, finite constants $a_3,\, a_4,\, a_5$ and $0<\alpha<1$, such that 
\begin{enumerate}
\item[\rm(i)] $a_3d(x,y)\le d'(x,y) \le a_4 d(x,y)$,
\item[\rm(ii)] for any $x,\, y,\, z\in X$ and $r>0$ 
\begin{equation*}
|d'(x,z)-d'(y,z)|\le a_5 r^{1-\alpha}(d'(x,y))^\alpha, \ 
\text{ provided } d'(x,z),\, d(y,z)<r.
\end{equation*}
\end{enumerate}
We set $B'(x,r)=\{y\in X; d'(x,y)<r\}$. Let $x\in X$ and $\varepsilon>0$. Using (ii) we see that for any $z\in X$ with $d'(x,z)<\varepsilon$ it holds $\lim\limits_{d'(y,x)\to 0} d'(y,z)<\varepsilon$, which implies $\mathbf{1}_{B'(y,\varepsilon)}\to \mathbf{1}_{B'(x,\varepsilon)}$ as $d'(x,y)$ tends to $0$. From this we see that $\mu(B'(x,\varepsilon))$ is continuous with respect to $x$ on $X$. 

We also easily check that 
\begin{equation*}
(2A')^{-1} f_{B'(x, a_3\varepsilon)}
\le f_{B(x,\varepsilon)} \le 2A' f_{B'(x,a_4\varepsilon)},
\end{equation*}
where $A'$ is a constant satisfying $\mu(B(x,a_4\varepsilon /a_3))\le A' \mu(B(x,\varepsilon))$, and $f_{B}=\fint_B f\, d\mu$.
Hence, to show Proposition \ref{p:appr-BMO}, we may assume that $d(x,y)$ satisfies the condition (ii) for $d'(x,y)$. 

To show Proposition \ref{p:appr-BMO}, we first note the following lemma. 
%%%%%%%%%%%%%%%%%%%%%%% Lemma Lemma Lemma %%%%%%%%%%%%%%%%%%%%%%%%%%%
\begin{lemma}\label{lem:approx-by-cptsuppBMO}
Let $f\in {\BMO}(X)$. Then for any $\varepsilon>0$, we have 
\begin{equation*}
f_\varepsilon(x):=\frac1{\mu(B(x,\varepsilon))} \int_{B(x,\varepsilon)}f\, d\mu \in {\BMO}(X)\cap C(X),
\end{equation*} 
and $\|f_\varepsilon\|_{{\BMO}(X)}\le c_1 \|f\|_{{\BMO}(X)}$, where $c_1=c_1(A,K)$ depends only on $A$ and $K$. 
\end{lemma}
%%%%%%%%%%%%%%%%%%%%%%% Lemma Lemma Lemma %%%%%%%%%%%%%%%%%%%%%%%%%%%

%%%%%%%%%%%%%%%%%%%%%%%% Proof Proof Proof %%%%%%%%%%%%%%%%%%%%%%%%%%%%%
\begin{proof}
The continuity of $f_\varepsilon$ is easily checked, by using the condition (ii). Thus, it suffices to check that 
$\|f_\varepsilon\|_{{\BMO}(X)}\le c_1 \|f\|_{{\BMO}(X)}$. To this end, we calculate the oscillation of $f_\varepsilon$ on an arbitrary ball $B(z,r)$. (The following calculations are due to Professor Eiichi Nakai.)

(i) The case $0<r\le \varepsilon $. For $x\in B(z,r)$ we have $B(x,\varepsilon)\subset B(z,K(r+\varepsilon))$ and 
\begin{align*}
|f_\varepsilon(x)-f_{B(z,\,K(r+\varepsilon))}|
&=\biggl|\frac{1}{\mu(B(x,\varepsilon))}\int_{B(x,\varepsilon)}f(y)\,d\mu(y)
-f_{B(z,\,K(r+\varepsilon))}\biggr|
\\
&\le \frac{1}{\mu(B(x,\varepsilon))}\int_{B(x,\varepsilon)}
|f(y)-f_{B(z,\,K(r+\varepsilon))}|\,d\mu(y)
\\
&\le \frac{1}{\mu(B(x,\varepsilon))}\int_{B(z,K(r+\varepsilon))}
|f(y)-f_{B(z,\,K(r+\varepsilon))}|\,d\mu(y)
\\
&\le \frac{\mu(B(z,K(r+\varepsilon)))}{\mu(B(x,\varepsilon))}
\|f\|_{{\BMO(X)}}.
\end{align*}
Since it holds $d(x,y)\le K(r+K(r+\varepsilon))\le K(1+2K)\varepsilon$ for $x\in B(z,r)$ and $y\in B(z,K(r+\varepsilon))$, we get 
${\mu(B(z,K(r+\varepsilon)))}/{\mu(B(x,\varepsilon))} \le A^{(\log_2K(1+2K))+1}$. Then it yields that 
\begin{equation*}
\frac{1}{\mu(B(z,r))}\int_{B(z,r)} |f_\varepsilon(x)-f_{B(z,\,K(r+\varepsilon))}|\,d\mu(x)
\le A^{(\log_2K(1+2K))+1}\|f\|_{{\BMO(X)}},
\end{equation*}
which shows 
\begin{equation*}
\|f_\varepsilon\|_{{\BMO}(X)} \le 2A^{(\log_2K(1+2K))+1}\|f\|_{{\BMO}(X)}.
\end{equation*}

(ii) The case $0<\varepsilon <r$. We note $B(x,\varepsilon)\subset B(z,K(r+\varepsilon))$ for $x\in B(z,r)$. Set $c=f_{B(z,K(r+\varepsilon))}$. Then we get
\begin{align*}
&\frac{1}{\mu(B(z,r))}\int_{B(z,r)}|f_\varepsilon(x)-c|\,d\mu(x)
\\
&\le \frac{1}{\mu(B(z,r))}\int_{B(z,r)}\biggl(
\frac{1}{\mu(B(x,\varepsilon))}\int_{B(x,\varepsilon)}|f(y)-c|d\mu(y)\,d\mu(x)
\\
&=\frac{1}{\mu(B(z,r))}\int_{X}\biggl(\int_{B(z,K(r+\varepsilon))}
\frac{\chi_{B(x,\varepsilon)}(y)}{\mu(B(x,\varepsilon))}|f(y)-c|d\mu(y)\biggr)
\chi_{B(z,r)}(x)\,d\mu(x)
\\
&=\frac{1}{\mu(B(z,r))}\int_{B(z,K(r+\varepsilon))}\biggl(\int_{X}
\frac{\chi_{B(y,\varepsilon)}(x)}{\mu(B(x,\varepsilon))}\chi_{B(z,r)}(x)\,
d\mu(x)\biggr)|f(y)-c|d\mu(y)
\\
&\le \frac{1}{\mu(B(z,r))}\int_{B(z,K(r+\varepsilon))}\biggl(
\int_{B(y,\varepsilon)}\frac{1}{\mu(B(x,\varepsilon))}\,d\mu(x)
\biggr)|f(y)-c|d\mu(y).
\end{align*}
For $x\in B(y,\varepsilon)$ we have $B(y,\varepsilon)\subset B(x,2K\varepsilon)$, and so we have $\mu(B(y,\varepsilon))\le A^{(\log_22K))+1}\mu(B(x,\varepsilon))$. Hence by $B(z,K(r+\epsilon)) \subset B(z,2Kr)$ we get
\begin{align*}
&\frac{1}{\mu(B(z,r))} \int_{B(z,r)} |f_{\epsilon}(x)-c|\,d\mu(x)
\\%%%%%%%%%
&\le A^{(\log_22K))+1}\frac{1}{\mu(B(z,r))} 
\int_{B(z,K(r+\varepsilon))} |f(y)-c|d\mu(y)
\\%%%%%%%%%
&\le A^{2(\log_22K)+2}\|f\|_{{\BMO}(X)},
\end{align*}
which implies
\begin{equation*}
\|f_\varepsilon\|_{{\BMO}(X)}
\le 2A^{2(\log_22K)+2}\|f\|_{{\BMO}(X)}.
\end{equation*}
This completes the proof. 
\end{proof}
%%%%%%%%%%%%%%%%%%%%%%%% End End End Proof %%%%%%%%%%%%%%%%%%%%%%%%%%%%

If $\mu(X)<\infty$, we know that for any $x_0\in X$, $X\subset B(x_0,R_0)$ for some $R_0>0$. By Lebesgue's differentiation 
theorem, we see that $\lim\limits_{\varepsilon\to0}f_\varepsilon (x)=f(x)$ for almost all $x\in X$. Hence, setting 
\begin{equation*}
f_j(x)= \frac{1}{\mu(B(x,1/j))} \int_{B(x,1/j)}f(y)\,d\mu(y),
\end{equation*} 
we see that $\{f_j\}$ satisfies the required condition for Proposition \ref{p:appr-BMO}, and so we have proved  Proposition  
\ref{p:appr-BMO} in this case. 

Next, we treat the case $\mu(X)=\infty$. Fix $x_0\in X$ and $S>0$ with $a_0=\mu(B(x_0,S))>0$. Let $K_0=2K$ and $A_0>1$ 
be such that $\mu(B(x_0,K_0 r))\le A_0\mu(B(x_0,r))$ for any $r>0$. For $k\in\N$, there exists a unique $\lambda_k\in\N$ such that 
\begin{equation}\label{eq:doubling0}
A_0^k a_0\le \mu(B(x_0,K_0^{\lambda_k}S))< A_0^{k+1}a_0.
\end{equation}
In fact, because of $\mu(X)=\infty$, there exists a unique $\lambda_k\in\N$ 
such that
\begin{equation*}
\mu(B(x_0,K_0^{\lambda_k-1}S))< A_0^{k}a_0 \le \mu(B(x_0,K_0^{\lambda_k}S)).
\end{equation*}
On the other hand, we have 
\begin{equation*}
\mu(B(x_0,K_0^{\lambda_k}S))\le A_0\mu(B(x_0,K_0^{\lambda_k-1}S)).
\end{equation*}
Hence we have
\begin{equation*}
A_0^k a_0\le \mu(B(x_0,K_0^{\lambda_k}S))< A_0^{k+1}a_0.
\end{equation*}
Note that $\lambda_1<\lambda_2<\dots$ and $\lim\limits_{k\to\infty}\lambda_k=\infty$. 

%%%%%%%%%%%%%%%%%%%%%%%% Lemma Lemma Lemma 1 %%%%%%%%%%%%%%%%%%%%%%%%%
\begin{lemma}\label{lem:Lemma1}
Let $\mu(X)=\infty$, $K_0$ and $A_0$ be as above. Set
\begin{equation*}
f_j(x)=\sum_{k=0}^{j}\mathbf{1}_{B(x_0,K_0^{\lambda_k}S)}.
\end{equation*}
Then there exists $c_0>0$ such that $\|f_j\|_{{\BMO}(X)}\le c_0$, where $\lambda_{0}=0$ and $c_0=c_0(A,K)$ depends only on $A$ and $K$. 
\end{lemma}
%%%%%%%%%%%%%%%%%%%%%%% End End End Lemma 1 %%%%%%%%%%%%%%%%%%%%%%%%%%%

%%%%%%%%%%%%%%%%%%%%%%%%% Proof Proof Proof %%%%%%%%%%%%%%%%%%%%%%%%%%%%
\begin{proof} 
Let $B:=B(x_1,R)$ be an arbitrary ball in $X$. For simplicity, write $B_{-1}=\emptyset$ and $B_{\ell} := B(x_0, K_0^{\lambda_\ell}S)$ for each $\ell \in \N$.  Our arguments depend upon the spatial position of $B$.\\ 
%%%%%%%%%%%%%%%%%%%%%%% (i) %%%%%%%%%%%%%%%%%%%%%%%%%%%%%%
\medskip\noindent
{\bf Case 1:} $B \subset B_0$, or $B \subset B_j^c$, or $B \subset B_{\ell} \setminus B_{\ell-1}$, $\ell=1,2,\dots,j$.  
In this case, $f_j(x)=\text{constant}$ on $B$, and hence 
\begin{align*}
{\MO}(f_j,B)=\fint_{B} \bigg|f_j(x) - \frac{1}{\mu(B)} \int_{B} f_j d\mu \bigg|d\mu(x)=0.
\end{align*}%%%%%%%%%%%%%%%%%%%%%%%%%%%%%%%%%%%%%%%%%%%%%
%%%%%%%%%%%%%%%%%%%%%%% (ii) %%%%%%%%%%%%%%%%%%%%%%%%%%%%%%
\medskip\noindent
{\bf Case 2:} $B \subset B{_\ell} \setminus B_{\ell-2}$,  $B \cap (B_{\ell} \setminus B_{\ell-1}) \neq \emptyset$ and 
$B \cap B_{\ell-1} \neq \emptyset$,  $\ell=1,\dots,j$. In this case we have $\MO(f_j,B) \le 1$. Indeed, if $\mu(B \cap (B_{\ell} \setminus B_{\ell-1})) \ge \mu(B \cap B_{\ell-1})$, then we have 
\begin{equation*}
\int_{B}|f_j(x)-(j-\ell+1)|\,d\mu(x) = \mu(B \cap B_{\ell-1}),
\end{equation*}
and so 
\begin{equation*}
\MO(f_j,B) \le 2\times \frac{1}{2}=1.
\end{equation*}
Otherwise, it holds  
\begin{equation*}
\int_{B} |f_j(x)-(j-\ell+2)|\,d\mu(x) =\mu(B \cap (B_{\ell} \setminus B_{\ell-1})),
\end{equation*}
and we have the same estimate. %%%%%%%%%%%%%%%%

%%%%%%%%%%%%%%%%%%%%%%% (iii) %%%%%%%%%%%%%%%%%%%%%%%%%%%%%%
\medskip\noindent
{\bf Case 3:} $B \subset B_{\ell}$, $B \cap (B_{\ell} \setminus B_{\ell-1}) \ne \emptyset$ and $B \cap B_{\ell-2} \ne \emptyset$, $\ell=2,\dots,j$. Taking $y_1\in B\cap(B_{\ell} \setminus B_{\ell-1})$ and $y_0 \in B \cap B_{\ell-2}$, one has 
\begin{equation*}
d(x_0,y_1) \le K(d(x_0,y_0)+d(y_0,y_1)) < K(K_0^{\lambda_{\ell-2}}S+d(y_0,y_1)),
\end{equation*}
from which we get 
\begin{equation*}
d(y_0,y_1)>\frac{d(x_0,y_1)}{K}-K_0^{\lambda_{\ell-2}}S.
\end{equation*}
On the other hand, we have $d(x_0,y_1)>K_0^{\lambda_{\ell-1}}S$ and 
\begin{equation*}
d(y_0,y_1)\le K(d(y_0,x_1)+d(x_1,y_1))<2KR. 
\end{equation*}
Hence, we obtain  
\begin{equation*}
2KR>\frac{K_0^{\lambda_{\ell-1}}S}{K}-K_0^{\lambda_{\ell-2}}S, 
\end{equation*}
and so,
\begin{align*}
R & > \frac{1}{2K} \Bigl(\frac{K_0^{\lambda_{\ell-1}}S}{K}
-K_0^{\lambda_{\ell-2}}S \Bigr)
\\%%%%%%%%%
&>\frac{1}{2K} \Bigl(\frac{K_0^{\lambda_{\ell-1}}S}{K}
-K_0^{\lambda_{\ell-1}-1}S \Bigr)
\\%%%%%%%%%%
&=K_0^{\lambda_{\ell-1}}\frac{K_0-K}{2K^2K_0}S
=\frac{K_0^{\lambda_{\ell-1}}S}{2KK_0}
=K_0^{\lambda_{\ell-1}-2}S.
\end{align*}
Now, if $y\in B(x_0,K_0^{\lambda_{\ell-1}}S))$, then there holds 
\begin{align*}
d(x_1,y)&\le K(d(x_1,y_0)+d(y_0,y)) < K(R+K(d(y_0,x_0)+d(x_0,y)))
\\%%%%%%%%%%%
&< K(R+K(K_0^{\lambda_{\ell-2}}S+K_0^{\lambda_{\ell-1}}S))
<K(R+2KK_0^{\lambda_{\ell-1}}S)
\\%%%%%%%%%
&< K(R+2K\cdot K_0^2 R) = K(1+K_0^3)R, 
\end{align*}
which implies that 
\begin{align*}
A_0^{\ell-1}a_0 \le \mu(B(x_0,K_0^{\lambda_{\ell-1}}S)) 
\le \mu(B(x_1,K(1+K_0^3)R)) 
\le A^{(\log_2K(1+K_0^3))+1}\mu(B(x_1,R)),
\end{align*}
that is,  $\mu(B(x_1,R))\ge CA_0^{\ell-1}a_0$. Consequently, we have
\begin{align*}
&\frac{1}{\mu(B)} \int_{B} |f_j(x)-(j-\ell+1)|\,d\mu(x)
\\%%%%%%%%%
&\le\frac{1}{\mu(B)} \sum_{i=0}^{\ell-1} \mu(B_i) 
\le \frac{1}{CA_0^{\ell-1}a_0} \sum_{i=0}^{\ell-1} A_0^{i+1} a_0 
<\frac{A_0^2}{C(A_0-1)}.
\end{align*}
This gives that 
\begin{equation*}
\MO(f_j,B) \le \frac{2A_0^2}{C(A_0-1)}.
\end{equation*}
%%%%%%%%%%%%%%%%%%%%%%% (iv) %%%%%%%%%%%%%%%%%%%%%%%%%%%%

\medskip\noindent
{\bf Case 4:} $B \subset B_{j+k}$, $B \cap (B_{j+k} \setminus B_{j+k-1}) \ne \emptyset$ and $B \cap B_{j-1} \ne \emptyset$, $k=2,\dots$. In this case, we have $\mu(B) \ge CA_0^{j+k-1}a_0$ as in Case 3, and hence
\begin{align*}
\fint_{B} |f_j(x)|\,d\mu(x)
&\le\frac{1}{\mu(B)} \sum_{i=0}^j \mu(B_i)
\le \frac{1}{CA_0^{j+k-1}a_0} \sum_{i=0}^j A_0^{i+1}a_0 
<\frac{A_0^{3-k}}{C(A_0-1)}.
\end{align*}
Therefore, it yields that 
\begin{equation*}
\MO(f_j,B)\le\frac{2A_0^{3-k}}{C(A_0-1)}.
\end{equation*}
Since $A_0>1$, we have completed the proof of our Lemma. 
\end{proof}
%%%%%%%%%%%%%%%%%%%% End End End Proof of Lemma 1 %%%%%%%%%%%%%%%%%%%%%%%%%%

%%%%%%%%%%%%%%%%%%%%%%%%% Lemma Lemma Lemma 2 %%%%%%%%%%%%%%%%%%%%%%%%
\begin{lemma}\label{lem:Lemma2}
Let $f\in {\BMO}(X)$. For $N>0$ define $[f]_N$ by 
\begin{equation*}
[f]_N(x)=\begin{cases}
f(x), &|f(x)|\le N,\\
N\frac{f(x)}{|f(x)|}, &|f(x)|> N.
\end{cases}
\end{equation*}
Then it holds $\|[f]_N\|_{{\BMO}(X)}\le 2\|f\|_{{\BMO}(X)}$. 
\end{lemma}
%%%%%%%%%%%%%%%%%%%%%%%%% End End End Lemma 2 %%%%%%%%%%%%%%%%%%%%%%%%%

For a proof, see for example \cite[p.~206]{Yabuta}. 
%%%%%%%%%%%%%%%% end Proof of Lemma 2 %%%%%%%%%%%%%%%%%%%%%%

%%%%%%%%%%%%%%%%%%%%%%%% Lemma Lemma Lemma 3 %%%%%%%%%%%%%%%%%%%%%%%%%
\begin{lemma}\label{lem:Lemma3}
For $f,\,g\in {\BMO}(X)\cap L^\infty(X)$, we have 
\begin{equation*}
\|fg\|_{{\BMO}(X)}
\le 2(\|f\|_{{\BMO}(X)}\|g\|_{\infty} 
+\|f\|_{\infty}\|g\|_{{\BMO}(X)}). 
\end{equation*}
\end{lemma}
%%%%%%%%%%%%%%%%%%%%%%%%% End End End Lemma 3 %%%%%%%%%%%%%%%%%%%%%%%%%

%%%%%%%%%%%%%%%%%%%%%%%%%% Proof Proof Proof %%%%%%%%%%%%%%%%%%%%%%%%%%%
\begin{proof}
Let $B$ be a ball in $X$. Then we deduce 
\begin{align*}
&\fint_B|f(x)g(x)-(fg)_B|d\mu(x) \leq 2\fint_B|f(x)g(x)-f_Bg_B|d\mu(x)
\\%%%%%%%%%%
&=2\fint_B|f(x)g(x)-f_Bg(x)+f_Bg(x)-f_Bg_B|d\mu(x)
\\%%%%%%%%%%
&\le 2\fint_B|f(x)-f_B||g(x)|d\mu(x) + 2\fint_B|f_B|\,|g(x)-g_B|d\mu(x)
\\%%%%%%%%%%
&\le 2\|f\|_{{\BMO}(X)}\|g\|_{\infty} 
+ 2\|f\|_{\infty}\|g\|_{{\BMO}(X)},
\end{align*}
from which we can deduce the desired estimate. 
\end{proof}
%%%%%%%%%%%%%%%%%%%% End End End Proof of Lemma 3 %%%%%%%%%%%%%%%%%%%%%%%%%%

%%%%%%%%%%%%%%%%%%%%%%%%% Lemma Lemma Lemma 4 %%%%%%%%%%%%%%%%%%%%%%%%
\begin{lemma}\label{lem:Lemma4}
Suppose $\mu(X)=\infty$. For any $f\in {\BMO}(X)$ there exists a sequence of bounded and boundedly supported $\{f_j\}$ such that 
\begin{align*}
&\|f_j\|_{{\BMO}(X)} \le 2 (2+c_0) \|f\|_{\BMO(X)},\\
&|f_j(x)|\le |f(x)|\ x \in X, \\  
& \lim_{j\to\infty}f_j(x)=f(x), \ x \in X,
\end{align*}
where $c_0$ is the constant defined in Lemma \ref{lem:Lemma1}. 
\end{lemma}
%%%%%%%%%%%%%%%%%%%%%%%%% End End End Lemma 4 %%%%%%%%%%%%%%%%%%%%%%%%%

%%%%%%%%%%%%%%%%%%%%%%%%%% Proof Proof Proof %%%%%%%%%%%%%%%%%%%%%%%%%%
\begin{proof}
To show this, we may assume that $\|f\|_{\BMO(X)}=1$ without loss of generality. Fix $x_0\in X$, and let $K_0=2K$ be as in Lemma \ref{lem:Lemma1}. Set 
\begin{equation*}
g_j(x)=\frac{1}{j+1} \sum_{k=0}^{j} \mathbf{1}_{B(x_0,K_0^{\lambda_k} j)}, \text{ and }f_j=[f]_jg_j.
\end{equation*}
Then we see that $g_j(x)=1$ on $B(x_0,j)$, and hence 
\begin{equation*}
|f_j(x)|\le |f(x)| \ \ \text{and } \ \lim_{j\to\infty}f_j(x)=f(x), \ x \in X.
\end{equation*}
By Lemmas \ref{lem:Lemma3}, \ref{lem:Lemma2} and \ref{lem:Lemma1}, we have 
\begin{align*}
\|f_j\|_{{\BMO}(X)}
&\le 2(\|[f]_j\|_{{\BMO}(X)}\|g_j\|_{\infty}
+\|[f]_j\|_{\infty}\|g_j\|_{{\BMO}(X)}),
\\
&\le 2\Bigl(2\|f\|_{{\BMO}(X)}+j\,\frac{c_0}{j}\Bigr)
=2(2+c_0)\|f\|_{\BMO(X)}.
\end{align*}
This proves Lemma \ref{lem:Lemma4}. 
\end{proof}
%%%%%%%%%%%%%%%%%%%%%% End End End Proof of Lemma 4 %%%%%%%%%%%%%%%%%%%%%%%%

Now, we proceed to show Proposition \ref{p:appr-BMO} in the case 
$\mu(X)=\infty$. Set
\begin{align*}
f_j := g_j  [h_j]_j  \ \text{ and } \ 
h_j(x) = \frac{1}{\mu(B(x,1/j))} \int_{B(x,1/j)}f(y)\,d\mu(y).
\end{align*}
Then we obtain 
\begin{align*}
&\|h_j\|_{{\BMO}(X)}\le c_0\|f\|_{{\BMO}(X)},\\
&|h_j(x)|\le \mathbb{M}f(x), \ x\in X,\\ 
&\lim_{j\to\infty}h_j(x)=f(x), \ \text{ a.e. } x \in X.
\end{align*}
So, by Lemma \ref{lem:Lemma4} we see that $f_j$ satisfies the desired condition in Proposition \ref{p:appr-BMO}. This completes the proof of Proposition \ref{p:appr-BMO}.  

%%%%%%%%%%%%%%%% Acknowledgements Acknowledgements Acknowledgements %%%%%%%%%%%%%%%%%
%%%%%%%%%%%%%%%% Acknowledgements Acknowledgements Acknowledgements %%%%%%%%%%%%%%%%%
\section*{Acknowledgements}
M. Cao would like to thank Professor Lixin Yan for his hospitality during working at Sun Yat-sen University where a part of this work was done, and thank Professor Dongyong Yang for helpful discussions. We also thank Professor Eiichi Nakai for useful discussions and several corrections in the last section.

%%%%%%%%%%%%%%%%%%%%%% Bibliography Bibliography Bibliography %%%%%%%%%%%%%%%%%%%%
%%%%%%%%%%%%%%%%%%%%%% Bibliography Bibliography Bibliography %%%%%%%%%%%%%%%%%%%%

\end{document}